\documentclass{article}

\makeatletter%
\@ifclassloaded{iopart}
{ 
\usepackage{iopams,amsthm}
\newcommand{\iopartonly}[1]{#1}
\newcommand{\otheronly}[1]{}
}
{ 
\usepackage{amsthm, amssymb, amsmath}

\newcommand{\iopartonly}[1]{}
\newcommand{\otheronly}[1]{#1}

}%
\makeatother
\newcommand{\english}{
\theoremstyle{plain}
\newtheorem{theorem}{Theorem}[section]
\newtheorem*{theorem*}{Theorem}
\newtheorem{lemma}[theorem]{Lemma}
\newtheorem*{lemma*}{Lemma}
\newtheorem{proposition}[theorem]{Proposition}
\newtheorem*{proposition*}{Proposition}
\newtheorem{corollary}[theorem]{Corollary}
\newtheorem*{corollary*}{Corollary}
\newtheorem{conjecture}[theorem]{Conjecture}
\newtheorem*{conjecture*}{Conjecture}
\theoremstyle{definition}
\newtheorem{definition}[theorem]{Definition}
\newtheorem*{definition*}{Definition}
\newtheorem{example}[theorem]{Example}
\newtheorem*{example*}{Example}
\newtheorem{exercise}[theorem]{Exercise}
\newtheorem*{exercise*}{Exercise}
\newtheorem{assumption}[theorem]{Assumption}
\newtheorem*{assumption*}{Assumption}
\newtheorem{notation}[theorem]{Notation}
\newtheorem*{notation*}{Notation}
\newtheorem{problem}[theorem]{Problem}
\newtheorem*{problem*}{Problem}
\theoremstyle{remark}
\newtheorem*{remark}{Remark}
\newtheorem*{claim}{Claim}
\newtheorem*{acknowledgement}{Acknowledgement}
}

\newcommand{\beq}{\begin{equation}}
\newcommand{\eeq}{\end{equation}}
\newcommand{\beqn}{\begin{equation}}
\newcommand{\eeqn}{\end{equation}}
\newcommand{\beqa}{\begin{eqnarray}}
\newcommand{\eeqa}{\end{eqnarray}}
\newcommand{\beqan}{\begin{eqnarray}}
\newcommand{\eeqan}{\end{eqnarray}}
\newcommand{\ben}{\begin{enumerate}}
\newcommand{\een}{\end{enumerate}}
\newcommand{\bit}{\begin{itemize}}
\newcommand{\eit}{\end{itemize}}
\newcommand{\bce}{\begin{center}}
\newcommand{\ece}{\end{center}}
\newcommand{\bsli}{\outlineonly{\begin{slide}}}
\newcommand{\esli}{\outlineonly{\end{slide}}}
\newcommand{\od}[2]{\frac{d{#1}}{d{#2}}}

\newcommand{\pd}[2]{\frac{\partial{#1}}{\partial{#2}}}

\newcommand{\RR}{\mathbb{R}}

\makeatletter%
\@ifclassloaded{iopart}{}{

}
\makeatother

\newcommand{\officialpaper}{
\def\isofficialpaper{}
\newcommand{\job}[1]{}
\newcommand{\future}[1]{}
\newcommand{\obs}[1]{}
\newcommand{\labelP}[1]{\label{##1}}
\newcommand{\question}[1]{}
\newcommand{\answer}[1]{}
\newcommand{\authoronly}[1]{}      
\newcommand{\officialonly}[1]{##1} 
\newcommand{\outlineonly}[1]{}     
}

\newcommand{\ignore}[1]{}

\usepackage{graphicx}
\usepackage{float}
\english
\officialpaper

\title{Absolutely continuous copulas with prescribed support constructed by differential equations, with an application in toxicology}

\author{Oscar Bj{\"o}rnham\\FOI CBRN Defence and Security\\{\tt oscar.bjornham@foi.se} \and
Niklas Br{\"a}nnstr{\"o}m\\FOI CBRN Defence and Security\\{\tt niklas.brannstrom@foi.se} \and
Leif Persson\\ Mathematics Department, Ume{\aa} University\\{\tt leif.persson@umu.se}}

\begin{document}

\maketitle
\begin{abstract}
A new method for constructing absolutely continuous two--dimensional copulas by differential equations is presented. The copulas are symmetric with respect to reflection in the opposite diagonal. The support of the copula density may be prescribed to arbitrary opposite symmetric hypographs of invertible functions, containing the diagonal. The method is applied to toxicological probit modeling, where new compatibility conditions for the probit parameters are derived. 
\end{abstract}

\section{Introduction and main results}\label{sec:Introduction}
This paper is motivated by the following result, which is probably well known, although we have not been able to find any explicit statement or proof:
\begin{proposition}\label{pro:linearConstraintStandardNormal}
	Suppose that $a,\Delta\in\RR, a>0$. Then there exists random variables $X,Y$ satisfying 
	\begin{equation}\label{eqn:linearConstraintStandardNormal}
		Y\leq a X + \Delta\text{ and }X, Y\text{ standard normal}
	\end{equation}
	if and only if $a=1$ and $\Delta\geq 0$, and then if $\Delta>0$, there exists $X,Y$ with absolutely continuous joint distribution satisfying (\ref{eqn:linearConstraintStandardNormal}). 
\end{proposition}
A proof is given at the end of this section. Our interests in this result comes from applications in toxicological probit modeling, accounted for in Section \ref{sec:Toxicology} where we prove new compatibility conditions for toxicological probit models. For simulation purposes, we are also interested in \emph{constructing} absolute continuous distributions of Proposition \ref{pro:linearConstraintStandardNormal}:
\begin{problem}\label{prb:standardNormalRestricted}
	Given a number $\Delta>0$, construct a pair of standard normal random variables $X,Y$ with absolutely continuous joint distribution supported on $y\leq x+\Delta$. 
\end{problem}
This seems to be a very simple and basic problem in probability theory, but to our surprise we could not find any simple constructions in the literature. Independent standard normal $X,Y$ have absolutely continuous joint distribution but do not fullfill the support condition, and truncating to $y\leq x+\Delta$ yields non--normal marginals. It is easy to construct singular solutions to the problem, the simplest being $X=Y$. The difficulty lies in imposing the absolute continuity.
We reduce Problem \ref{prb:standardNormalRestricted} to a problem of the dependence structure, or \emph{copula} of $(X,Y)$. Before we state our main result, let us briefly review the main facts about copulas.

A function $C:[0,1]^2\to[0,1]$ is said to be a copula if $C(u,0)=C(0,v)=0$, $C(u,1)=u$, $C(1,v)=v$ and $C(u_2,v_2)-C(u_2,v_1)-C(u_1,v_2)+C(u_1,v_1)\geq 0$ for all $u,v,u_1,v_1,u_1,v_2\in[0,1]$ such that $u_1\leq u_2, v_1\leq v_2$, cf. \cite[Definition 2.2.2]{Nelsen1999}.
By Sklar's theorem (\cite[Theorem 2.3.3]{Nelsen1999}), the cumulative distribution function (CDF) $F_{X,Y}$ of any bivariate random variable $(X,Y)$ is representable by the marginal CDF's $F_X, F_Y$ and a copula $C$ as
\begin{equation}\label{eqn:Copula}
	F_{X,Y}(x,y)=C(F_X(x),F_Y(y)).
\end{equation}
This may be regarded as a change of variables $X=F^{-1}_X(U), Y=F^{-1}_Y(V)$ such that $(U,V)$ has uniform marginals. The copula $C$ is uniquely defined on Range$(F_X)\times$Range$(F_Y)$ for all bivariate random variables $(X,Y)$, and if $F_X,F_Y$ are continuous, $C$ is uniquely defined on $[0,1]^2$. Morover, the partial derivatives $C'_u,C'_v,C''_{uv}$ of a copula $C(u,v)$ are defined almost everywhere on $[0,1]^2$ (\cite[Theorem 2.2.7]{Nelsen1999}) and $C''_{uv}\geq 0$. If $\iint C''_{uv}dudv=1$, $C$ is said to be \emph{absolutely continuous}. Copulas are common in statistical modeling, in particular mathematical finance. The main benefit of copulas is that by Sklar's theorem, the marginal statistics and dependence structure can be modeled separately. For an introduction to copulas we refer to \cite{Nelsen1999}, for a recent review see \cite{Flores_etal2017}.

Returning to Problem \ref{prb:standardNormalRestricted}, the half--plane $\{(x,y):y\leq x+\Delta\} $ is symmetric with respect to reflection $(x,y)\mapsto(-y,-x)$ through the line $x+y=0$. Therefore, we assume that $(X,Y)$ and $(-Y,-X)$ are equal in distribution. Moreover, $X,-X,Y,-Y$ are all identically distributed so it follows (from Theorem \ref{thm:oppositeRadialSymmetry} below) that the copula $C(u,v)$ of $(X,Y)$ is opposite symmetric, according to the following definition.
\begin{definition}\label{def:oppsiteSymmetryC}
	A copula $C$ is said to be \emph{opposite symmetric} if
	\begin{equation}\label{eqn:oppositeSymmetryC}
	C(u,v)=C(1-v,1-u)+u+v-1
\end{equation}
for all $(u,v)\in[0,1]^2$.
\end{definition}
Opposite symmetry means symmetry with respect to reflection $(u,v)\mapsto(1-v,1-u)$ in the opposite diagonal $u+v=1$, and was introduced in \cite{DeBaetsDeMeyerUbeda-Flores2009}. Applying the copula transformation, using the standard normal CDF $\Phi$:
 \begin{equation}
 	u=\Phi(x), v=\Phi(y), F_{X,Y}(x,y)=C(u,v),{}
 \end{equation}
 Problem \ref{prb:standardNormalRestricted} reduces to finding an absolutely continuous opposite symmetric copula $C(u,v)$ with density supported on $\{(u,v)\in[0,1]^2:v\leq H(u)\}$ where
 \begin{equation}\label{eqn:GaussianH}
 	H(u)= \Phi(\Phi^{-1}(u)+\Delta).
 \end{equation}
 Our main result is the construction of $C(u,v)$ in the following Theorem \ref{thm:main}. We want to emphasize its simplicity, involving $H$ and its inverse \emph{explicitly}. The crucial part is the evaluation of the integral in (\ref{eqn:mainG}), which is suitable for numerical integration if not analytically integrable.
\begin{theorem}\label{thm:main}
	Assume that $H:[0,1]\to[0,1]$ is a bijective function such that
	\begin{equation}\label{eqn:Hcond}
		H(u)+H^{-1}(1-u)=1\text{ and } H(u)\geq u,\quad u\in[0,1],
	\end{equation}
	\begin{equation}\label{eqn:u0cond}
		u_0\in(0,1/2)\text{ and }H(u_0)=1-u_0,
	\end{equation}
	\begin{equation}\label{eqn:mainHcond1}
		\int_{u_0}^u \frac{dz}{H(z)-z}<\infty,\quad u\in [u_0,1)
	\end{equation}
	and
	\begin{equation}\label{eqn:mainHcond2}
		\lim_{u\nearrow 1}\int_{u_0}^u \frac{dz}{H(z)-z}=\infty.
	\end{equation}
	Let
	\begin{equation}\label{eqn:mainG}
		G(v)=\exp\left(-\int_{u_0}^{1-v}\frac{dz}{H(z)-z}\right), \quad v\in[0,1-u_0],
	\end{equation}
	\begin{equation}\label{eqn:mainK}
		K(u)=\frac{H(u)-u}{G(1-u)}-1+2u_0, \quad u\in[u_0,1]
	\end{equation}
	and
	\begin{equation}\label{eqn:mainF}
		F(u)=(1-2u_0)(1-G(1-u)),\quad u\in[u_0,1].
	\end{equation}
	Define $C(u,v)$ by
	\begin{enumerate}
 		\item If $0<u\leq u_0$ and $0\leq v\leq H(u)$, then
 		\begin{equation}\label{eqn:mainC1}
 			C(u,v)=H^{-1}(v)+(K(1-v)-K(1-H(u)))G(v).
 		\end{equation}
 		\item If $0<u\leq u_0$ and $H(u)<v\leq 1-u$ then 
 		\begin{equation}\label{eqn:mainC2}
 			C(u,v)=u
 		\end{equation}
 		\item If $u_0<u<1$ and $0\leq v\leq 1-u$ then
 		\begin{equation}\label{eqn:mainC3}
 			C(u,v)=H^{-1}(v)+(K(1-v)+F(u))G(v).
 		\end{equation}
 		\item If $0<u<1$ and $u+v>1$ then $C(u,v)$ is defined by (\ref{eqn:oppositeSymmetryC}).
 	\end{enumerate}
 	Then $C(u,v)$ is an absolutely continuous opposite symmetric copula with probability density supported on $v\leq H(u)$.
\end{theorem}
Note that the hypograph $v\leq H(u)$ is opposite symmetric if and only if (\ref{eqn:Hcond}) holds true. The copula is piecewisely defined, on parts of the unit square depicted in Figure \ref{fig:SquarePartsMain}.
\begin{figure}[H]\label{fig:SquarePartsMain}
	\includegraphics[width=0.8\textwidth]{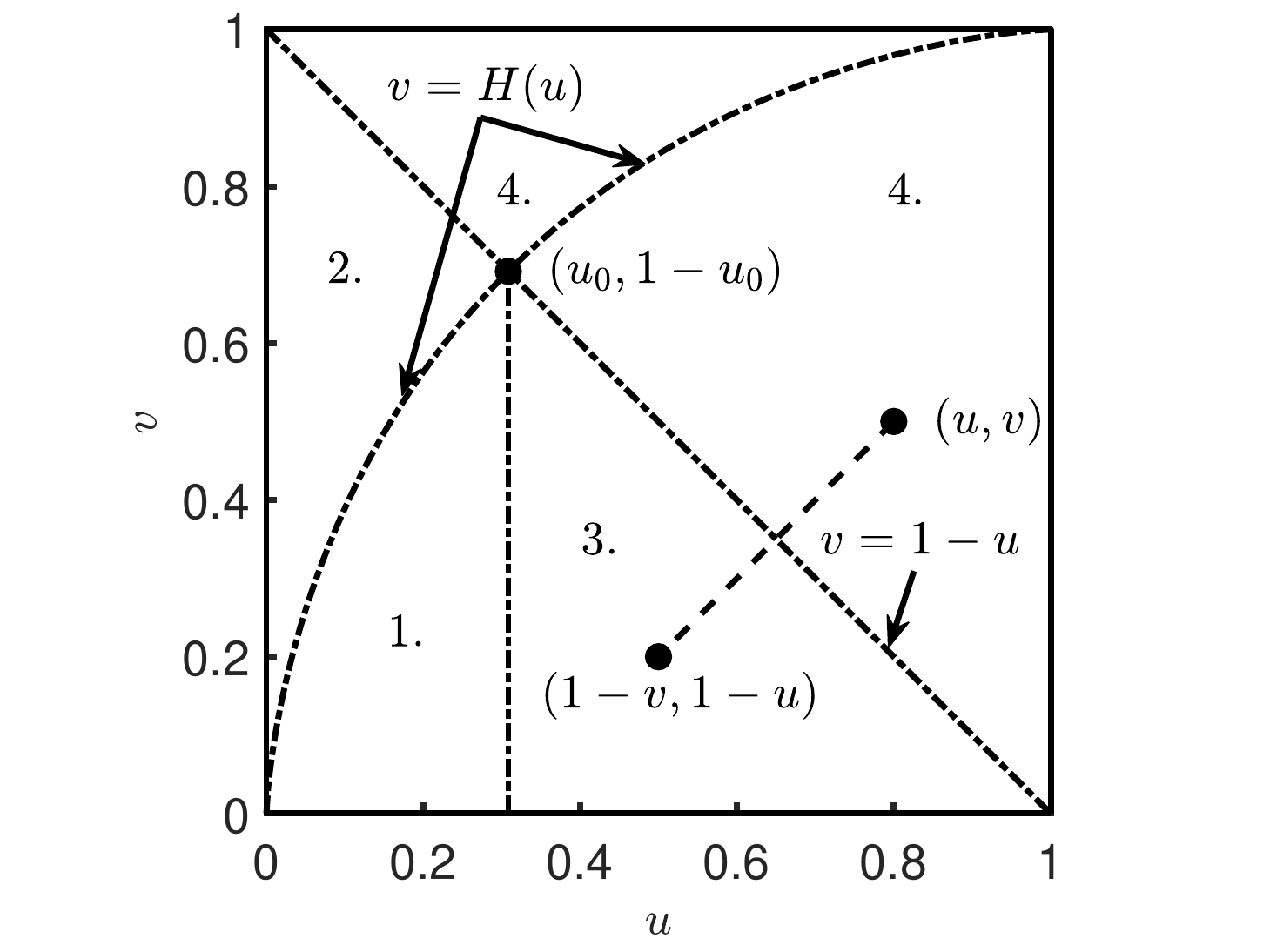}
	\caption{Parts of the unit square for piecewise definition of the copula in Theorem \ref{thm:main}.}
\end{figure}
Theorem \ref{thm:main} is proved at then end of Section \ref{sec:prescribed}. Before that, we develop a theory for construction of opposite symmetric copulas by differential equations in Section \ref{sec:DifferentialEquations} and Section \ref{sec:prescribed}, which we believe is of interest in its own right, and gives in fact a much larger class of copulas than Theorem \ref{thm:main}. In section \ref{sec:Comparison} we compare our method to two other methods in the literature, Durantes and Jaworskis construction of absolutely continuous copulas with given diagonal section \cite{DuranteJaworski2008}, and Jaworskis characterization of copulas using differential equations \cite{Jaworski2014}. In Section \ref{sec:sampling} we adapt our differential equation method to sampling from the copula. We conclude the paper with section \ref{sec:Toxicology}, an application in toxicological probit modeling, where new compatibility conditions for the probit coefficients are derived. 
\begin{example}\label{exm:Gaussian}
	In this example we construct a solution to Problem \ref{prb:standardNormalRestricted} using Theorem \ref{thm:main}. Let $\Phi$ be the standard normal CDF, $\phi(x)=\Phi'(x)$ the standard normal probability density function (PDF), $\Delta>0$ and $H$ given by (\ref{eqn:GaussianH}). Then $H^{-1}(v)=\Phi(\Phi^{-1}(v)-\Delta)$
	and because of the symmetries $\Phi(x)+\Phi(-x)=1$, $\Phi^{-1}(u)+\Phi^{-1}(1-u)=0$, condition (\ref{eqn:Hcond}) is satisfied, and 
	\begin{equation}
		u_0=\Phi(-\Delta/2).
	\end{equation}
	Moreover, with the change of variables $z=\Phi(w)$ and the mean value theorem we obtain
	\begin{multline}
		\int_{u_0}^{u}\frac{dz}{H(z)-z}=\int_{-\Delta/2}^{\Phi^{-1}(u)}\frac{\phi(w)dw}{\Phi(w+\Delta)-\Phi(w)}\\
		=\int_{-\Delta/2}^{\Phi^{-1}(u)}\frac{\phi(w)dw}{\phi(w+\theta(w)\Delta)}
		=\frac1{\sqrt{2\pi}}\int_{-\Delta/2}^{\Phi^{-1}(u)}\exp\left(w\Delta\theta(w)-\frac{\Delta^2\theta(w)^2}2\right)dw
	\end{multline}
	for some function $\theta(w)$ with $0\leq\theta(w)\leq1$, so
	\begin{equation}
		\frac1{\sqrt{2\pi}\Delta}\left(e^{\Delta\Phi^{-1}(u)}-e^{\Delta^2/2}\right)\geq\int_{u_0}^{u}\frac{dz}{H(z)-z}\geq \frac{e^{-\Delta^2/2}}{\sqrt{2\pi}}\left(\Phi^{-1}(u)+\frac{\Delta}2\right)
	\end{equation}
	which proves that conditions (\ref{eqn:mainHcond1}) and (\ref{eqn:mainHcond2}) are satisfied. The function $G$ defined by equation (\ref{eqn:mainG}) can not be expressed in terms of special functions (to our knowledge), but can be determined by numerical integration, and $C(u,v)$ is then determined by equations (\ref{eqn:oppositeSymmetryC}) and (\ref{eqn:mainC1})-(\ref{eqn:mainC3}). The density of $C$ is illustrated in figure \ref{fig:Cuv}. The joint PDF of $(X,Y)$ is given by
	\begin{equation}
		p(x,y)=C''_{uv}(\Phi(x),\Phi(y))\phi(x)\phi(y)
	\end{equation}
	and is illustrated in figure \ref{fig:p}. Here, $G(v)$ is computed with the  MATLAB\textsuperscript{\textregistered} function {\tt integral } at $400$ uniformly distributed grid points on $[\epsilon, 1-u_0]$, and computed at intermediate points on $[\epsilon,1-u_0]$ by spline interpolation, where $\epsilon=10^{-11}$. Consequently, the copula and its density is computed on $[\epsilon, 1-\epsilon]^2$. 
	\begin{figure}[H]\label{fig:Cuv}
	$\begin{array}{rl}
	\includegraphics[width=0.5\textwidth]{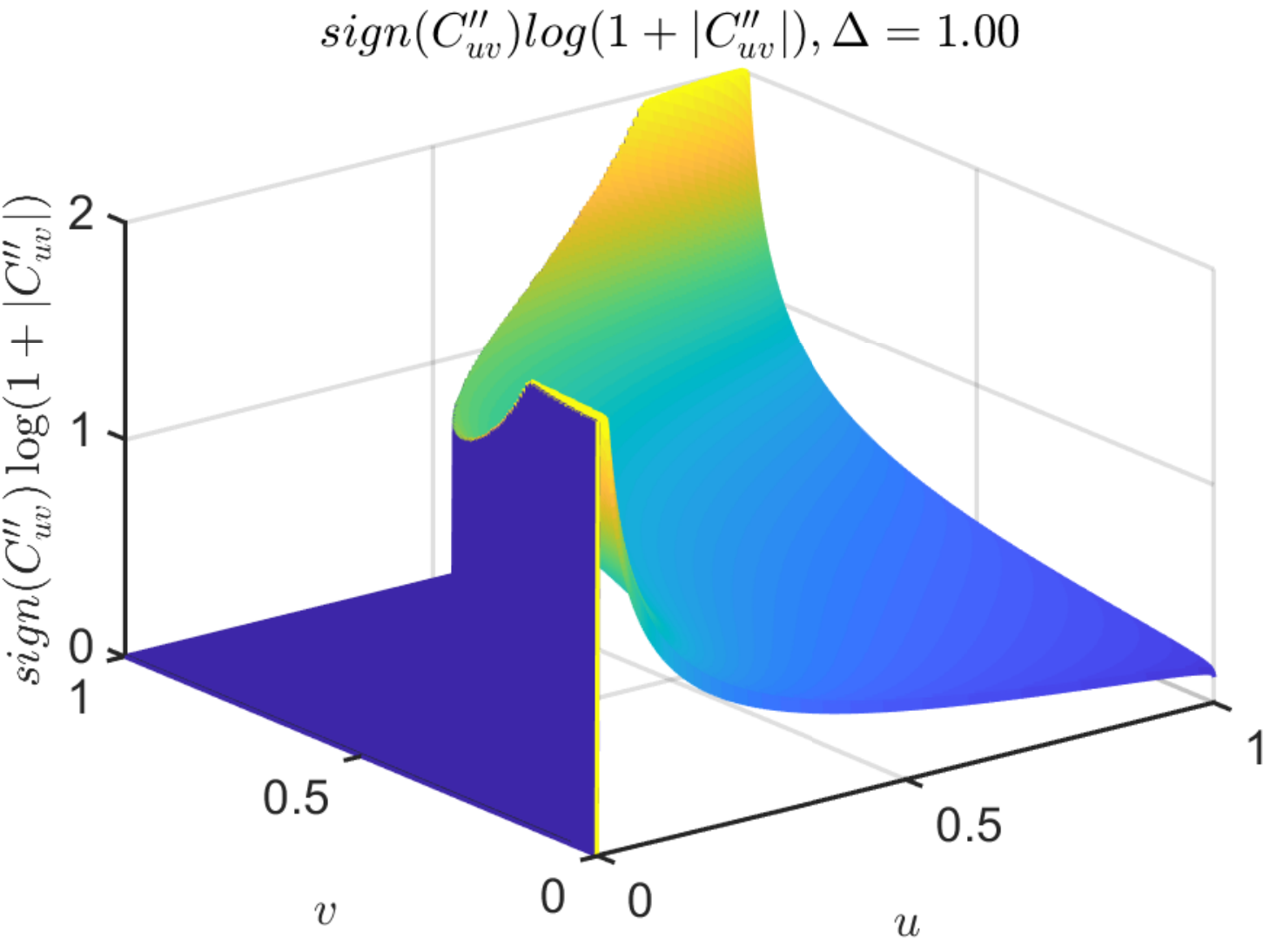}&
	\includegraphics[width=0.5\textwidth]{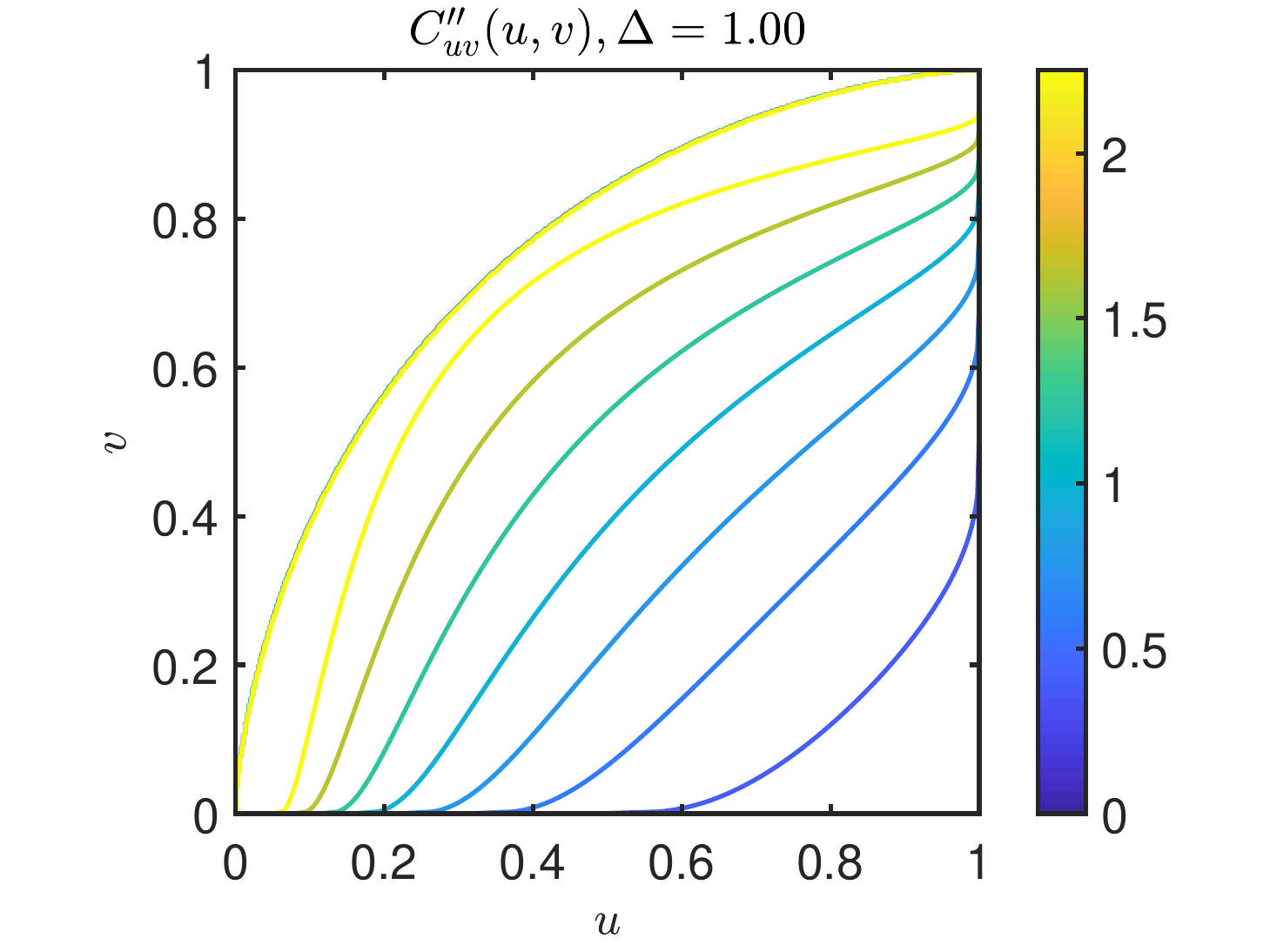}
	\end{array}$
	\caption{Copula density $C''_{uv}(u,v)$ for Example \ref{exm:Gaussian}, $\Delta=1$. The density is discontinuous on the curve $v=H(u)$ and tends to infinity when approaching $(0,0)$ or $(1,1)$.}
	\end{figure}
	\begin{figure}[H]\label{fig:p}
	$\begin{array}{rl}
	\includegraphics[width=0.5\textwidth]{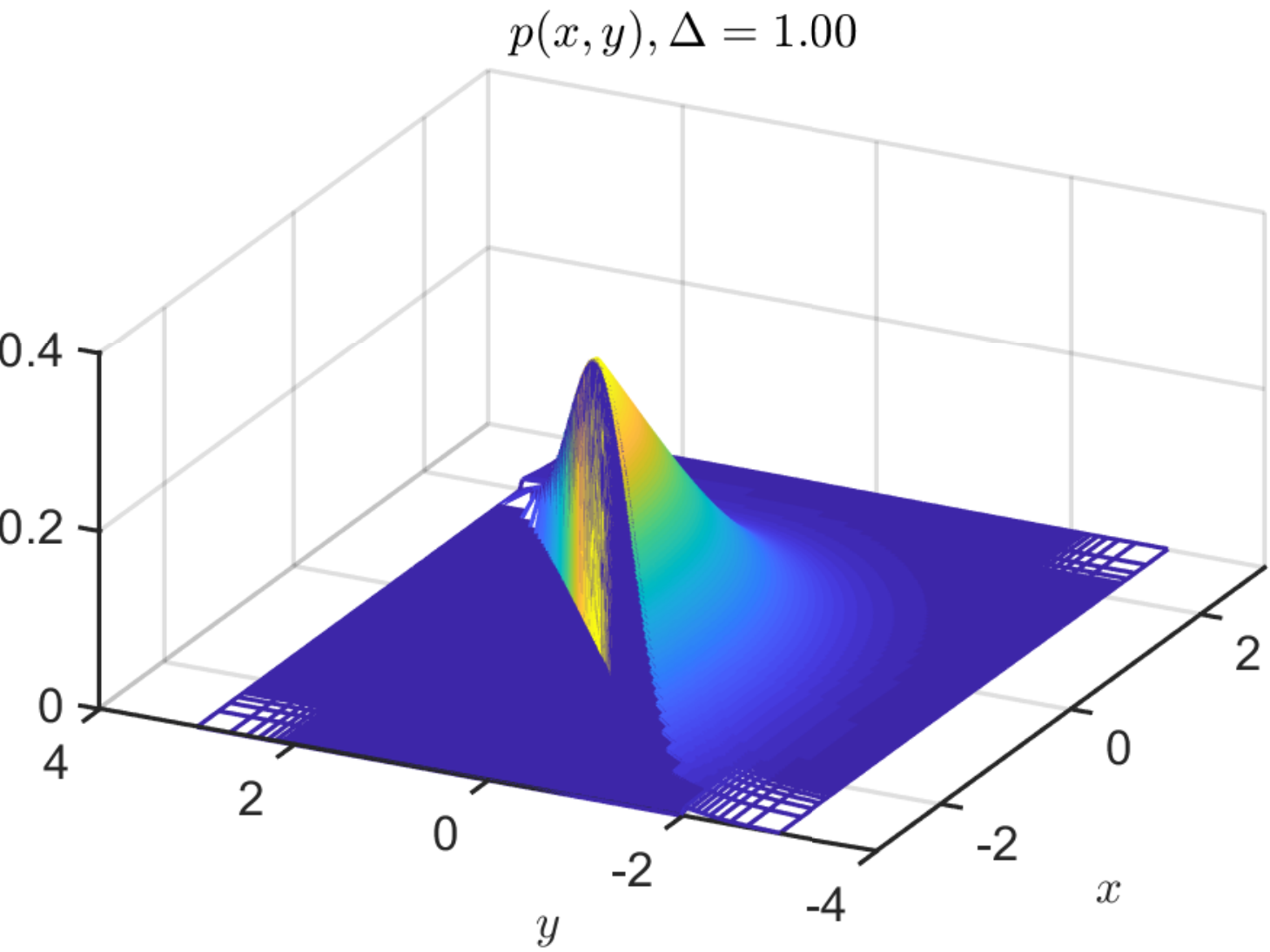}&
	\includegraphics[width=0.5\textwidth]{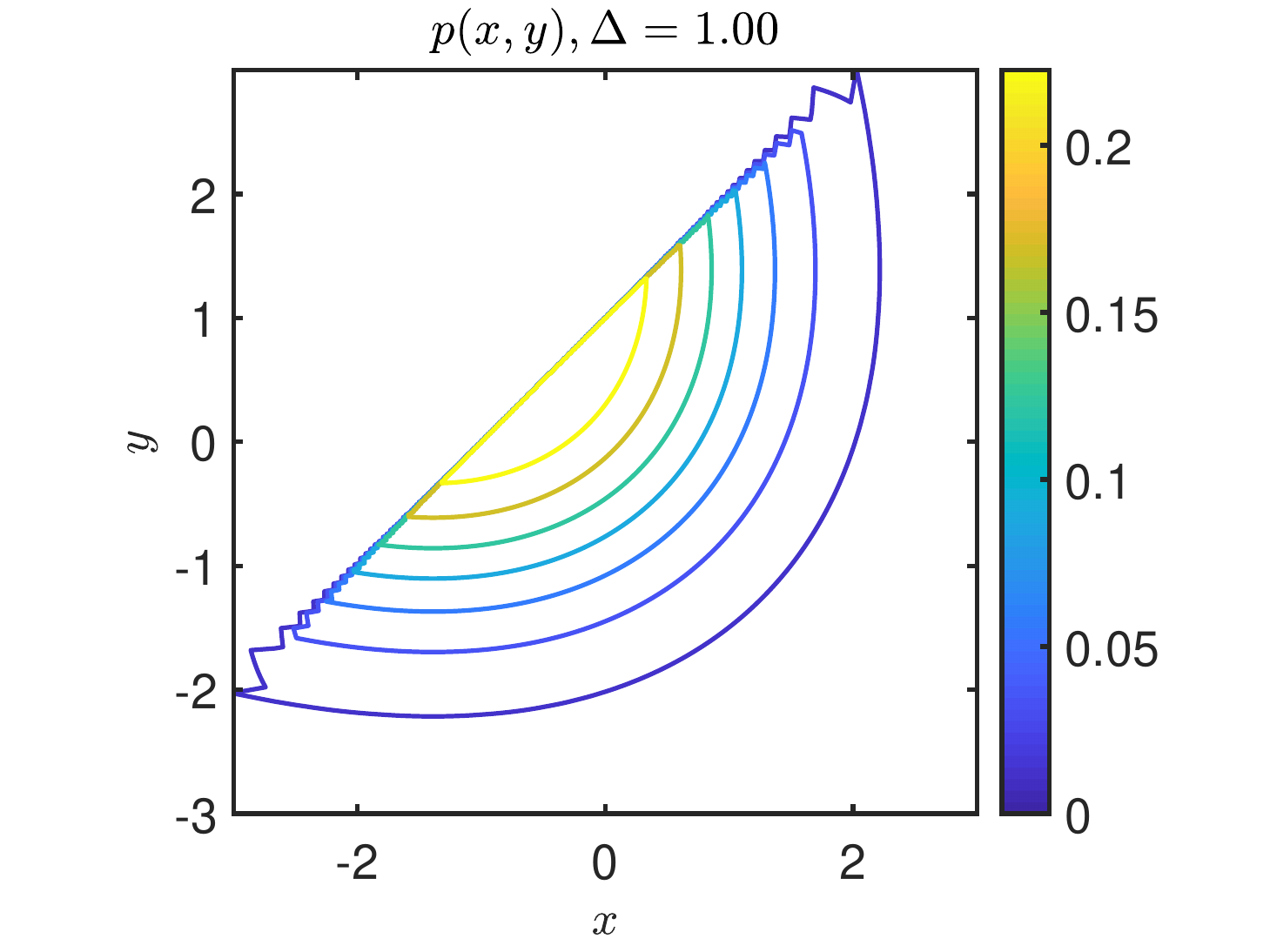}
	\end{array}$
	\caption{Probability density function $p(x,y)$ for Example \ref{exm:Gaussian}, $\Delta=1$. The wiggles in the level curves at the upper right and lower left corners of right plot are numerical artifacts.}
	\end{figure}
\end{example}
\begin{proof}[Proof of Proposition \ref{pro:linearConstraintStandardNormal}]
If (\ref{eqn:linearConstraintStandardNormal}) is satisfied then 
$\Phi((y-\Delta)/a)=P\{aX+\Delta\leq y\}\leq P\{Y\leq y\}=\Phi(y)$ 
for all $y\in\RR$, which is possible only if $a=1$ and $\Delta\geq 0$. For $\Delta\geq 0$ we can take $X=Y$, which gives a singular distribution supported on $x=y$. If $\Delta>0$, Example \ref{exm:Gaussian} shows that $X,Y$ with absolutely continuous joint distribution exists. 
\end{proof}
\section{Symmetries and copulas}\label{sec:bivariateSymmetry} 
Several notions of bivariate symmetries are considered in \cite{Nelsen1993}. A pair of random variables$(X,Y)$ are said to be \emph{exchangeable} if $(X,Y)$ and $(Y,X)$ are equal in distribution, and $(X,Y)$ is exchangeable if and only if its copula $C(u,v)$ is a symmetric function, i.e., $C(u,v)=C(v,u)$. Moreover, $(X,Y)$ is said to be \emph{radially symmetric} about $(a,b)\in\RR^2$ if $(X-a,Y-b)$ and $(a-X,b-Y)$ are equal in distribution, or equivalently,
\begin{equation}\label{eqn:radialSymmetry}
	F_{X,Y}(a+x,b+y)=1-F_X(a-x)-F_Y(b-y)+F_{X,Y}(a-x,b-y)
\end{equation}
Also, $(X,Y)$ is said to be \emph{marginally symmetric} about $(a,b)\in\RR^2$ if 
\begin{equation}\label{eqn:marginalSymmetry}
	F_X(a+x)=1-F_X(a-x)\text{ and }F_Y(b+y)=1-F_Y(b-y).
\end{equation}
The following theorem is proved in \cite[Theorem 3.2]{Nelsen1993}:
\begin{theorem}\label{thm:radialSymmetry}
	Suppose $(X,Y)$ is marginally symmetric about $(a,b)$ with copula $C$. Then $(X,Y)$ is radially symmetric about $(a,b)$ if and only if $C$ satisfies the functional equation
	\begin{equation}\label{eqn:radialSymmetricC}
		C(u,v)=C(1-u,1-v)+u+v-1
	\end{equation}
\end{theorem}
There is a corresponding class of bivariate random variables associated to opposite symmetric copulas, which we propose to call \emph{opposite radially symmetric} variables, in accordance with the terminology in \cite{DeBaetsDeMeyerUbeda-Flores2009}, and analogous to the radially symmetric variables of \cite{Nelsen1993}.
\begin{definition}\label{def:oppositeRadialSymmetry}
	The bivariate random variable $(X,Y)$ is said to be \emph{opposite radially symmetric} about $(a,b)\in\RR^2$ if $(a+X,b+Y)$ and $(b-Y,a-X)$ are equal in distribution, or, equivalently,
	\begin{equation}\label{eqn:oppositeRadialSymmetry}
	 	F_{X,Y}(a+x,b+y)=1-F_X(a-y)-F_Y(b-x)+F_{X,Y}(a-y,b-x).
	 \end{equation} 
\end{definition}
\noindent We need to replace marginal symmetry with the following analog of (\ref{eqn:radialSymmetry}):
\begin{definition}\label{def:oppositeMarginalSymmetry}
	The bivariate random variable $(X,Y)$ is said to be \emph{opposite marginally symmetric} about $(a,b)\in\RR^2$ if $F_X,F_Y$ satisfy
	\begin{equation}\label{eqn:oppositeMarginalSymmetry}
		F_X(a+x)=1-F_Y(b-x)\text{ and }F_Y(b+y)=1-F_X(a-y)
	\end{equation}
	for all $x,y$.
\end{definition}
\begin{remark}\label{exm:oppositeMarginalSymmetry}
	If $X,Y$ are identically distributed and marginally symmetric about $(a,a)\in\RR^2$, then $(X,Y)$ is opposite marginally symmetric about $(a,a)$. There are no identically distributed opposite marginally symmetric  $(X,Y)$ about $(a,b)$ if $b\neq a$, since then the common CDF $F_X=F_Y=F$ would satisfy $F(x)=F(x+b-a)$ for all $x$.
\end{remark}
\noindent We have the following analog of Theorem \ref{thm:radialSymmetry}:
\begin{theorem}\label{thm:oppositeRadialSymmetry}
	Suppose that $(X,Y)$ is opposite marginally symmetric about $(a,b)\in\RR^2$ with copula $C$, and suppose that $F_X$, $F_Y$ are  continuous. Then $(X,Y)$ is opposite radially symmetric about $(a,b)$ if and only if $C$ is opposite symmetric.
\end{theorem}
\begin{proof}
	It follows from equations (\ref{eqn:oppositeRadialSymmetry}) and (\ref{eqn:oppositeMarginalSymmetry}) that $(X,Y)$ is opposite radially symmetric if and only if 
	\begin{multline}
		C(1-F_Y(b-x),1-F_X(a-y))=
		C(F_X(a+x),F_Y(b+y))\\
		=1-F_X(a-y)-F_Y(b-x)+C(F_X(a-y),F_Y(b-x)).
	\end{multline}
	Since the range of $F_X$ and $F_Y$ is $[0,1]$ this proves the theorem.
\end{proof}
\begin{remark}
	There is an erroneous statement in \cite[Remark 1]{DeBaetsDeMeyerUbeda-Flores2009} that if $C$ is opposite symmetric, then $(X,Y)$ and $(1-Y,1-X)$ are equal in distribution, i.e., $(X,Y)$ is opposite radially symmetric about $(1/2,1/2)$, but additional assumptions like opposite marginal symmetry in Theorem \ref{thm:oppositeRadialSymmetry} is needed to draw that conclusion.
\end{remark}

 \section{Differential equations for copulas with opposite symmetry}\label{sec:DifferentialEquations}
 The following theorem provides a characterization of absolutely continuous copulas with opposite symmetry, and constitutes  the basis for deriving the differential equations. We also obtain a simple formula for Kendall's $\tau$ rank correlation coefficient for opposite symmetric copulas. Kendall's $\tau$ is defined as $\tau_C=-1+4\int_0^1\int_0^1 C(u,v)dC(u,v)$, cf \cite[chapter 5]{Nelsen1999}. 
\begin{theorem}\label{thm:copulaCsymmetry}
	Assume that $p$ is an integrable function on $[0,1]^2$satisfying
	\begin{equation}\label{eqn:oppositeSymmetryP}
		p(u,v)=p(1-v,1-u)
	\end{equation}
	and let
	\begin{equation}\label{eqn:defC}
 		C(u,v)=\int_0^u\int_0^v p(w,z)dwdz.
 	\end{equation}
	Then 
	\begin{equation}\label{eqn:oppositeSymmetryGeneralC}
		C(u,v)=C(1-v,1-u)+C(u,1)+C(1,v)-C(1,1)
	\end{equation}
	and the following two conditions are equivalent:
	\begin{description}
		\item[1.] $C'_u(u,1)=1$ for all $u\in[0,1]$.
		\item[2.] $C'_v(1,v)=1$ for all $v\in[0,1]$.
		\end{description}
		Furthermore, if $p\geq 0$ these conditions are equivalent to
		\begin{description}
			\item[3.] $C$ is an absolutely continuous opposite symmetric copula.
		\end{description}
		and then if also
		\begin{equation}\label{eqn:IntCuCvFinite}
			\int_0^1\int_0^1 C'_u C'_vdudv<\infty,
		\end{equation}
		Kendall's $\tau$ is given by
		\begin{equation}\label{eqn:oppositeKendallTau}
			\tau_C=-1+8\int_0^1 C(u,1-u)du 
		\end{equation}
\end{theorem}
\begin{proof}[Proof of Theorem \ref{thm:copulaCsymmetry}]
	By the inclusion-exclusion principle for integrals we have 
	\begin{equation}\label{eqn:pIntegral1}
		\int_u^1\int_v^1p(w,z)dzdw=C(u,v)+C(1,1)-C(u,1)-C(1,v).
	\end{equation}
	By change of variables and symmetry (\ref{eqn:oppositeSymmetryP}) we also have
	\begin{multline}\label{eqn:pIntegral2}
		\int_u^1\int_v^1p(w,z)dzdw=\int_0^{1-v}\int_0^{1-u}p(1-z,1-w)dzdw\\
		=\int_0^{1-v}\int_0^{1-u}p(w,z)dzdw = C(1-v,1-u).
	\end{multline}
	which proves (\ref{eqn:oppositeSymmetryGeneralC}). Assume that $C'_u(u,1)=1$ for $u\in[0,1]$, it follows that $C(u,1)=u$ for $u\in[0,1]$. Then (\ref{eqn:oppositeSymmetryGeneralC}) with $u=0$ simplifies to $0=C(1,v)-v$, so $C'_v(1,v)=1$. Similarly, $C'_v(1,v)\equiv1\implies C'_u(u,1)\equiv 1$. If these conditions hold, $C(u,1)\equiv u$ and $C(1,v)\equiv v$, which shows that $C$ is a copula, which is absolutely continuous by equation (\ref{eqn:defC}), and equation (\ref{eqn:oppositeSymmetryGeneralC}) implies equation (\ref{eqn:oppositeSymmetryC}), i.e., opposite symmetry. Conversely, assuming $C$ an absolute continuous copula satisfying (\ref{eqn:oppositeSymmetryC}), differentiation yields $C'_u(u,1)\equiv 1$ and $C'_v(1,v)\equiv 1$. Suppose in addition that (\ref{eqn:IntCuCvFinite}) holds true. Differentiation of (\ref{eqn:oppositeSymmetryC}) yields
	\begin{equation}
		C'_u(1-v,1-u)=1-C'_v(u,v),\quad C'_v(1-v,1-u)=1-C'_u(u,v)
	\end{equation}
	which gives
	\begin{multline}
		\int_0^1\int_0^1 C'_uC'_vdvdu = \int_0^1\int_0^{1-u}C'_uC'_vdvdu + \int_0^1\int_{1-u}^1 C'_uC'_v dv du\\
		=\int_0^1\int_0^{1-u}C'_uC'_v+(1-C'_v)(1-C'_u)dvdu\\
		=\int_0^1\int_0^{1-u}1-C'_u-C'_v dvdu =\frac12 - 2\int_0^1 C(u,1-u)du
	\end{multline}
	According to \cite[equation (5.1.10)]{Nelsen1999}, equation (\ref{eqn:IntCuCvFinite}) implies that $\tau_C = 1-4\int_0^1\int_0^1C'_uC'_vdudv$, which proves (\ref{eqn:oppositeKendallTau}).
\end{proof}

We will now show that copulas satisfying the assumptions in Theorem \ref{thm:oppositeRadialSymmetry}, with the additional assumption of being \emph{conditionally independent on $u+v\leq 1$}  can be characterized by differential equations. This method is reminiscent of the well known method of \emph{separation of variables} for construction of solutions to partial differential equations. 
 This will also give a construction method for absolutely continuous copulas with given opposite diagonal section, a problem considered in \cite{DeBaetsDeMeyerUbeda-Flores2009}, cf. Theorem \ref{thm:separable3} below.  Later, we will modify the construction, restricting the copula density support to $v\leq H(u)$, which is required to solve Problem \ref{prb:standardNormalRestricted}.
\begin{theorem}\label{thm:separable}
	Assume that
	\begin{equation}\label{eqn:separableP}
		p(u,v)=\left\{
			\begin{matrix}
				F'(u)G'(v)&\text{ if }&u+v\leq 1\\
				F'(1-v)G'(1-u)&\text{ if }&u+v>1
			\end{matrix}
		\right.
	\end{equation}
	where $F(0)=G(0)=0$, $G'\geq 0$ and $C$ is given by (\ref{eqn:defC}). Then 
	\begin{equation}\label{eqn:separableCu}
		C'_u(u,v)=\left\{
			\begin{matrix}
				F'(u)G(v)&\text{ if }&u+v\leq 1\\
				G(1-u)F'(u)+G'(1-u)(F(u)-F(1-v))&\text{ if }&u+v>1
			\end{matrix}
		\right.
	\end{equation}
	and the following are equivalent:
	\begin{enumerate}
		\item $F'\geq0$ and 
		\begin{equation}\label{eqn:separableODE}
			G(1-u)F'(u)+G'(1-u)F(u)=1, u\in[0,1]
		\end{equation}
		\item $C(u,v)$ is an absolutely continuous copula,
	\end{enumerate} 
	and then
	\begin{equation}\label{eqn:separableC}
		C(u,v)=\left\{
			\begin{matrix}
				F(u)G(v)&\text{ if }&u+v\leq 1\\
				F(1-v)G(1-u)+u+v-1&\text{ if }&u+v>1
			\end{matrix}
		\right.
	\end{equation}
\end{theorem}
\begin{proof}
 Integration $C'_u(u,v)=\int_0^v C''_{uv}(u,z)dz$ of the piecewise defined function $p=C''_{uv}$ yields $C'_u(u,v)=F'(u)G(v)$ for $u+v\leq 1$ and $C'_u(u,v)=G(1-u)F'(u)+G'(1-u)(F(u)-F(1-v))$ for $u+v>1$, so $C'_u(u,1)=G(1-u)F'(u)+G'(1-u)F(u)$. Suppose that $F'\geq 0$ and (\ref{eqn:separableODE}) holds true. Then $p\geq 0$ and $C'_u(u,1)\equiv 1$ so $C$ is an absolutely continuous copula by Theorem \ref{thm:copulaCsymmetry}. Conversely, suppose that $C$ is an absolutely continuous copula. Then $C''_{uv}=p\geq 0$ so $F'\geq 0$ by (\ref{eqn:separableP}), and (\ref{eqn:separableODE}) holds since $C'_u(u,1)\equiv 1$. Moreover, integration $C(u,v)=\int_0^u C'_u(z,v)dz$ yields (\ref{eqn:separableC}) for $u+v\leq 1$, and (\ref{eqn:separableC}) for $u+v> 1$ follows from Theorem \ref{thm:copulaCsymmetry}.
\end{proof}
The differential equation (\ref{eqn:separableODE}) can be solved with the integrating factor method. Moreover, a condition for $F'(u)\geq 0$ can be derived.
\begin{theorem}\label{thm:separable2}
	Assume that $G$ satisfies the assumptions of Theorem \ref{thm:separable}.
	Then $F(u)$ satisfy (\ref{eqn:separableODE}) and $F(0)=0$ if and only if
	\begin{equation}\label{eqn:separableF}
		F(u)=G(1-u)\int_0^u\frac{dz}{G(1-z)^2}
	\end{equation}
	Moreover, if $F(u)$ is given by (\ref{eqn:separableF}), then
	\begin{equation}\label{eqn:separableFprime}
		F'(u)=G'(1-u)\left(\frac{L(0)}{G(1)^2}+\int_0^u \frac{1+L'(z)}{G(1-z)^2}dz\right)
	\end{equation}
	where
	\begin{equation}\label{eqn:L}
		L(u)=\frac{G(1-u)}{G'(1-u)}.
	\end{equation}
	Finally, if $u^*\in[0,1]$, $L'(u)\geq -1$ for $u\in(u^*,1)$ and if
	\begin{equation}\label{eqn:separablePosCond}
		-\int_0^{u^*} \frac{1+L'(z)}{G(1-z)^2}dz\leq \frac{L(0)}{G(1)^2}
	\end{equation}
	then $F'(u)\geq0$ for $u\in(0,1)$. 
\end{theorem}
\begin{proof}
	Equation (\ref{eqn:separableF}) is obtained by multiplying (\ref{eqn:separableODE}) with the integrating factor $1/G(1-u)^2$. Equation (\ref{eqn:separableODE}) yields
	\begin{equation}\label{eqn:separableODE2}
		F'(u)=\frac1{G(1-u)}-\frac1{L(u)}F(u)
	\end{equation}
	 and substituting (\ref{eqn:separableF}) in (\ref{eqn:separableODE2}) using (\ref{eqn:L}) yields 
	 \begin{equation}\label{eqn:separableFprime2}
		F'(u)=G'(1-u)\left(\frac{L(u)}{G(1-u)^2}-\int_0^u \frac{dz}{G(1-z)^2}\right)
	\end{equation}
	and the identity
	\begin{equation}\label{eqn:LbyGsquared}
		\od{}{u}\left(\frac{L(u)}{G(1-u)^2}\right)=\frac{2+L'(u)}{G(1-u)^2}
	\end{equation}
	yields 
	\begin{equation}
		\frac{L(u)}{G(1-u)^2}=\frac{L(0)}{G(1)^2}+\int_0^u\frac{2+L'(z)}{G(1-z)^2}dz
	\end{equation}
	 which proves (\ref{eqn:separableFprime}). Moreover, by the assumptions, $u\mapsto -\int_0^u (1+L'(z))/G(1-z)^2 dz$ has its maximum for $u=u^*$, so it follows from (\ref{eqn:separablePosCond}) that $F'(u)\geq F'(u^*)\geq 0$ for $u\in[0,1]$.
\end{proof}
\begin{example}
	$G(v)=v$, $L(u)=1-u$, $1+L'(u)=0$, $F'(u)=G'(1-u)/G(1)$, yields the independence copula $C(u,v)=uv$.
\end{example}
\begin{example}\label{exm:separable1}
	If $k\geq 1$ and $G(v)=v^k$, then (\ref{eqn:separableODE}) has solution
	\begin{equation}
		F(u)=\frac{(1-u)^{1-k}-(1-u)^k}{2k-1}
	\end{equation}
	and $F'(u)\geq 0$ for $u\in[0,1]$, so 
	\begin{equation}
		C(u,v)=\left\{
			\begin{matrix}
				((1-u)^{1-k}-(1-u)^k)v^k/(2k-1)&\text{ if }&u+v\leq 1\\
				(1-u)^k(v^{1-k}-v^k)/(2k-1)+u+v-1&\text{ if }&u+v>1
			\end{matrix}
			\right.
		\end{equation}
	is a one--parameter family of absolutely continuous copulas. In particular, for $k=1$ we obtain the independence copula $uv$. For $k>1$, $\lim_{u\nearrow1}F(u)=\infty$.
\end{example}
\begin{example}\label{exm:separable2}
	If $G(v)=\sin(\pi v/2)$, then (\ref{eqn:separableODE}) has solution
	\begin{equation}
		F(u)=2\sin(\pi u/2)/\pi
	\end{equation}
	and $F'(u)\geq 0$ for $u\in[0,1]$, so 
	\begin{equation}
		C(u,v)=\left\{
			\begin{matrix}
				2\sin(\pi u/2)\sin(\pi v/2)/\pi&\text{ if }&u+v\leq 1\\
				2\cos(\pi u/2)\cos(\pi v/2)/\pi+u+v-1&\text{ if }&u+v>1
			\end{matrix}
			\right.
		\end{equation}
	is an absolutely continuous copula.
\end{example}
Since the positivity conditions in Theorem \ref{thm:separable2} is formulated in terms of the function $L$, it is natural to start by specifying $L$ satisfying (\ref{eqn:separablePosCond}). This is also related to the problem of constructing copulas with prescribed \emph{opposite diagonal section} $\omega(u)=C(u,1-u)$ considered in \cite{DeBaetsDeMeyerUbeda-Flores2009}. In fact, given $\omega$, the function $L$ is given by the explicit formula (\ref{eqn:separableLfromOmega}) below. This is formulated in Theorem \ref{thm:separable3}.
\begin{theorem}\label{thm:separable3}
	Suppose that $L$ is a positive real--valued function defined on $[0,1]$ such that
	\begin{equation}\label{eqn:separable3condL1}
		\int_0^u \frac{dz}{L(z)}<\infty
	\end{equation}
	for $u\in[0,1)$ and
	\begin{equation}\label{eqn:separable3condL2}
		\lim_{u\to 1-}\int_0^u \frac{dz}{L(z)}=\infty.
	\end{equation}
	Let
	\begin{equation}\label{eqn:separableGfromL}
		G(v)=\exp\left(-\int_0^{1-v}\frac{dz}{L(z)}\right)
	\end{equation}
	and suppose that (\ref{eqn:separablePosCond}) holds true. Moreover, let $F(u)$ be given by (\ref{eqn:separableF}). Then $C$ given by (\ref{eqn:separableC}) is an absolutely continuous copula. Moreover, the opposite diagonal section
	\begin{equation}\label{eqn:separableOppositeDiagonal}
		\omega(u)\equiv C(u,1-u)
	\end{equation}
	satisfies
	\begin{equation}\label{eqn:separableLfromOmega}
		L(u)=\frac{2\omega(u)}{1-\omega'(u)}.
	\end{equation}
\end{theorem}
\begin{proof}
	Clearly, because $L$ is positive and satisfies (\ref{eqn:separable3condL1}) and (\ref{eqn:separable3condL2}), $G$ defined by (\ref{eqn:separableGfromL}) is positive, $G$ is increasing (in fact strictly increasing) and $G(0)=0$. Moreover, it follows from (\ref{eqn:separableGfromL}) that (\ref{eqn:L}) holds true. By Theorem \ref{thm:separable2}, $F'(u)\geq 0$ and by Theorem \ref{thm:separable}, $C$ is an absolutely continuous copula. Differentiation of $F(u)G(1-u)=\omega(u)$ yields $F'(u)G(1-u)-F(u)G'(1-u)=\omega'(u)$, so in view of (\ref{eqn:separableODE}) we get 
	\begin{equation}\label{eqn:separable3FprimeG}
		F'(u)G(1-u)=\frac{1+\omega'(u)}2
	\end{equation}
	and
	\begin{equation}\label{eqn:separable3FGprime}
		F(u)G'(1-u)=\frac{1-\omega'(u)}2
	\end{equation}
	Solving for $F(u)$ in (\ref{eqn:separable3FGprime}), differentiating and substituting $F'(u)$ in the left hand side of (\ref{eqn:separable3FprimeG}) yields
	\begin{equation}
		(1-\omega'(u))\frac{G''(1-u)G(1-u)}{G'(1-u)^2}-\omega''(u)\frac{G(1-u)}{G'(1-u)}=1+\omega'(u).
	\end{equation}
	Using (\ref{eqn:L}) and the identity
	\begin{equation}\label{eqn:Lprime}
		\frac{G''(1-u)G(1-u)}{G'(1-u)^2}=1+L'(u)
	\end{equation}
	we get
	\begin{equation}
		(1-\omega'(u))L'(u)-\omega''(u)L(u)=2\omega'(u)
	\end{equation}
	which is integrated to $(1-\omega'(u))L(u)=2\omega(u)+$constant. Since $\omega(1)=C(1,0)=0$ and $L(1)=0$ in view of (\ref{eqn:separable3condL2}), the integration constant is zero, which proves (\ref{eqn:separableLfromOmega}).
\end{proof}

\begin{example}
 	Assume that $k\geq 1$ and let $L(u)=(1-u)/k$. Then we get $G(1-u)=(1-u)^k$ so we recover Example \ref{exm:separable1}. Also, $L'(u)=-1/k\geq -1$ so $u^*=0$ and and since $0\leq L(0)=1/k$ we infer from Theorem \ref{thm:separable2} that an absolutely continuous copula is obtained.  
 \end{example} 
 \begin{example}
 	Assume that $a\in[0,1)$ and let $L(u)=(1-u)(1-au)$. Then
 	\begin{equation*}
 		G(1-u)=\left(\frac{1-u}{1-au}\right)^{1/(1-a)}
 	\end{equation*}
 	and $u^*=1/2$: $L'(u)=-1+a(-1+2u)\leq -1$ if $u\leq 1/2$, $L'(u)\geq -1$ if $u\geq 1/2$. We obtain 
 	\begin{multline*}
 		\int_0^{u^*}\frac{1+L'(z)}{G(1-z)^2}dz = \int_0^{1/2}\left(\frac{1-au}{1-u}\right)^{2/(1-a)}a(1-2u)du\\
 		=\frac12F_1\left(1,\frac2{1-a},-\frac2{1-a};3;\frac12,\frac{a}2\right)
 	\end{multline*}
 	Here $F_1$ is the Appell series (see \cite[p. 1027]{GradshteynRyzhik2014} for a definition), which may be represented by Picard's integral formula, cf. \cite{CuytDriverTanVerdonk1999}:
 	\begin{multline*}
 		F_1(a,b,b';c;x,y)\\=\frac{\Gamma(c)}{\Gamma(a)\Gamma(c-a)}\int_0^1t^{a-1}(1-t)^{c-a-1}(1-tx)^{-b}(1-ty)^{-b'}dt
 	\end{multline*}
 	Here, $\Gamma$ denotes Euler's gamma function (\cite[p. 901]{GradshteynRyzhik2014}). The function $F_1$ is available in computer algebra systems like Maple\textsuperscript{\textregistered} and Mathematica\textsuperscript{\textregistered}, and numerical investigation reveals that the right hand side is an increasing function of $a$ and approaches the value $0.861485$ as $a\to 1-$. Therefore condition (\ref{eqn:separablePosCond}) is satisfied, so Theorem \ref{thm:separable2} yields an absolutely continuous copula, and (\ref{eqn:separableF}) can be evaluated to
 	\begin{equation*}
 	 	F(u)=u G(1-u)F_1\left(1,\frac2{1-a},-\frac2{1-a};2;au,u\right).
 	 \end{equation*}
 	 When $2/(1-a)$ is integer, this expression can be simplified to a finite sum of powers and logarithms, cf. \cite{CuytDriverTanVerdonk1999}. 
 \end{example}
  
\section{Comparison with other methods}\label{sec:Comparison}
A method by Durante and Jaworski is found in \cite{DuranteJaworski2008}, where absolutely continuous copulas $C(u,v)$ with given \emph{diagonal section} $C(t,t)$ are constructed, in terms of convex combinations of singular \emph{diagonal copulas}
 \begin{equation}
 	C_{\delta}(u,v)=\min\left(u,v,\frac{\delta(u)+\delta(v)}2\right)
 \end{equation}
(satisfying $C_{\delta}(t,t)=\delta(t)$). The problem with this approach for our purposes is that the constraint $v\leq H(u)$ imposes functional inequalities $\delta(H(u))+\delta(u)\leq 2u$ that must be fullfilled for the $\delta$'s used in the construction. In comparison, the advantage of our differential equation method is that $H$ is used explicitly, using only elementary calculus. 

Regarding copulas and differential equations, there is a characterization of {\emph all} copulas by Jaworski, in terms of a certain type of weak solutions to differential equations in \cite{Jaworski2014}. For comparison we give here a simplified account of his method in the special case of absolutely continuous copulas with differentiable density and sectional inverse. For fixed $u\in[0,1]$ let $C(u,\cdot)^{-1}(z)$ denote the assumed unique solution $v$ to the equation $C(u,v)=z$, i.e., $C(u,C(u,\cdot)^{-1}(z))=z$ for all $z\in[0,1]$, and define
\begin{equation}
	C_{[u]}(t,z)=u^{-1} C(ut,C(u,\cdot)^{-1}(uz))
\end{equation}
Moreover, define
\begin{equation}
	F_C(u,z)=\left.\pd{}{t}C_{[u]}(t,z)\right|_{t=1}-z=C'_u(u,C(u,\cdot)^{-1}(uz))-z
\end{equation}
Now suppose that for each $v\in[0,1]$, $g_v(u)$ is solution to the terminal value problem
\begin{eqnarray}
	ug'_v(u)&=&F_C(u,g_v(u)),u\in(0,1)\label{eqn:JaworskiODE}\\
	g_v(1)&=&v\label{eqn:JaworskiTC}
\end{eqnarray}
Then $C$ can be characterized in terms of $g_v(u)$ as
\begin{equation}\label{eqn:JaworskiCharacterization}
	C(u,v)=u g_u(v)
\end{equation}
To see this, note that by the definition of $F_C$ and the product rule of differentiation, (\ref{eqn:JaworskiODE}) is equivalent to
\begin{equation}
	\od{}{u}(ug_v(u))=C'_u(u,C(u,\cdot)^{-1}(ug_v(u)))
\end{equation}
and this ODE for $g_v(u)$ is satisfied for $g_v(u)=C(u,v)/u$, so by uniqueness of solution to (\ref{eqn:JaworskiODE})-(\ref{eqn:JaworskiTC}), (\ref{eqn:JaworskiCharacterization}) must hold. The general result (valid for all copulas) can be found in \cite[Theorems 3.1 and 3.2]{Jaworski2014}. Now, applying Jaworski's characterization theorem to a copula of the form (\ref{eqn:separableC}), we need to compute $C(u,\cdot)^{-1}(z)$ to obtain $F_C$. For $z\leq 1-u$ we get $F(u)G(v)=z$, which can be solved explicitly, yielding $v=C(u,\cdot)^{-1}(z)=G^{-1}(z/F(u))$. However, for $z>1-u$, $v=C(u,\cdot)^{-1}(z)$ is implicitly defined by $F(1-v)G(1-u)+u+v-1=z$, which can not be solved for $v$ in terms of $F,G$ and their inverses. Therefore, we have not been able to use Jaworski's method to obtain equations for $F,G$ for copulas of the type (\ref{eqn:separableC}). 

\section{Absolutely continuous copulas with prescribed support}\label{sec:prescribed}
Here we construct absolutely continuous opposite symmetric copulas with the support of the probability measure prescribed by a constraint $v\leq H(v)$. The construction is simple, using elementary calculus and a piecewise definition of the copula density, similar to Theorem \ref{thm:separable}
 \begin{theorem}\label{thm:prescribed}
 	Suppose that $0<u_0<1/2$ and that $H$ is a strictly increasing function defined on $[0,1]$, continuously differentiable on $(0,u_0)$, satisfying $H(u_0)=1-u_0$ and satisfying the symmetry condition
 	\begin{equation}\label{eqn:symmetricH}
 		H(u)+H^{-1}(1-u)=1.
 	\end{equation}
 	Furthermore, suppose that $F$ is a differentiable function defined on $[u_0,1)$ such that $F(u_0)=0$, $G$ is a differentiable function defined on $[0,1-u_0]$ such that $G(0)=0$, $G'\geq 0$ and $C(u,v)$ given by (\ref{eqn:defC}), where
 	\begin{equation}\label{eqn:prescribedP}
 		p(u,v)=\left\{
 		\begin{matrix}
 			G'(v)/G(H(u))&\text{ if }&0<u\leq u_0, &0<v\leq H(u)\\
 			0&\text{ if }&0<u\leq u_0, &H(u)<v\leq 1-u\\
 			F'(u)G'(v)&\text{ if }&u_0<u<1, &0<v\leq 1-u\\
 			p(1-v,1-u)&\text{ if }& 0<u<1, &1-u<v<1
 		\end{matrix}
 		\right.
 	\end{equation}
 	Furthermore, let
 	\begin{equation}\label{eqn:K}
 		K(u)=\int_{u_0}^u \frac{H'(z)dz}{G(1-z)}
 	\end{equation}
 	for $u_0\leq u\leq1$. Then the following are equivalent:
 	\begin{enumerate}
 		\item $F'\geq 0$ and
 		\begin{equation}\label{eqn:prescribedODE}
 			F'(u)G(1-u)+G'(1-u)(F(u)+K(u))=1
 		\end{equation}
 		for $u\in[u_0,1)$.
 		\item $C(u,v)$ is an absolutely continuous copula, and then
 		\begin{enumerate}
 			\item If $0\leq u\leq u_0$ and $0\leq v\leq H(u)$, then
 			\begin{equation}\label{eqn:prescribedC1}
 				C(u,v)=H^{-1}(v)+(K(1-v)-K(H^{-1}(1-u)))G(v)
 			\end{equation}
 			\item If $0\leq u\leq u_0$ and $H(u)\leq v\leq 1-u$ then 
 			\begin{equation}\label{eqn:prescribedC2}
 				C(u,v)=u
 			\end{equation}
 			\item If $u_0\leq u\leq 1$ and $0\leq v\leq 1-u$ then
 			\begin{equation}\label{eqn:prescribedC3}
 				C(u,v)=H^{-1}(v)+(K(1-v)+F(u))G(v)
 			\end{equation}
 			\item If $0\leq u\leq 1$ and $u+v>1$ then $C(u,v)$ is given by (\ref{eqn:oppositeSymmetryC}).
 		\end{enumerate}
 	\end{enumerate}
 \end{theorem}

 \begin{proof}
 The basic idea of the proof is similar to Theorem \ref{thm:separable}: integrate the given piecewise defined \emph{ansatz} for the copula density $C''_{uv}$ to derive $C'_u$ and use Theorem \ref{thm:copulaCsymmetry}. By definition $p(u,v)=C''_{uv}(u,v)$ and piecewisely defined on the regions 1-7 depicted in Figure \ref{fig:SquarePartsPrescribed} as follows; region 1: $C''_{uv}=G'(v)/G(H(u))$, region 2,3,7: $C''_{uv}=0$, region 4: $C''_{uv}=F'(u)G'(v)$, region 5: $C''_{uv}=F'(1-v)G'(1-u)$, and region 6: $C''_{uv}=G'(1-u)/G(H(1-v))$. Integration yields $C'_u(u,v)=\int_0^vC''_{uv}(u,z)dz$, piecewisely defined as follows; region 1: $C'_u=G(v)/G(H(u))$, region 2,3: $C'_u=1$, region 4: $C'_u=F'(u)G(v)$, region 5: $C'_u=F'(u)G(1-u)+(F(u)-F(1-v))G'(1-u)$, region 6: $C'_u=F'(u)G(1-u)+(F(u)+K(H^{-1}(v)))G'(1-u)${}, and region 7: $C'_u=F'(u)G(1-u)+(F(u)+K(u))G'(1-u)$. To derive the expression in region 6, write $K$ on the alternate form
 \begin{equation}\label{eqn:Kb}
 	K(u)=\int_{H^{-1}(1-u)}^{u_0}\frac{dw}{G(H(w))}
 \end{equation}
 (derived by the change of variables $z=1-H(w)=H^{-1}(1-w)$) and note that 
 \begin{multline*}
 	\int_{1-u_0}^v\frac{dz}{G(H(1-z))}=\int_{1-v}^{u_0}\frac{dw}{G(H(w))}\\
 	=\int_{H^{-1}(1-H^{-1}(v))}^{u_0}\frac{dw}{G(H(w))}=K(H^{-1}(v))
 \end{multline*}
 in view of (\ref{eqn:symmetricH}). If $F'\geq 0$ and (\ref{eqn:prescribedODE}) holds true, then $p\geq 0$ by (\ref{eqn:prescribedP}) and $C'_u(u,1)\equiv1$ by (\ref{eqn:prescribedODE}) since the left hand side of (\ref{eqn:prescribedODE}) is the expression for $C'_u$ in region 7. Thus, by theorem \ref{thm:copulaCsymmetry}, $C$ is an absolutely continous copula. Conversely, if $C$ is an absolutely continuous copula, then $p=C''_{uv}\geq 0$ so $F'\geq 0$ and by (\ref{eqn:prescribedP}), and $C'_u(u,1)=1$ which proves (\ref{eqn:prescribedODE}). The conditions $C'_u(u,1)\equiv 1$ and $C'_v(1,v)\equiv 1$ are equivalent by Theorem \ref{thm:copulaCsymmetry}. Assume now that $C$ is an absolutely continuous copula, then $C'_u=1$ in region 7 by (\ref{eqn:prescribedODE}). Integration $C(u,v)=\int_0^u C'_u(z,v)dz$ yields the following piecewise defined function $C(u,v)$; region 2,3,7: $C=u$ which proves (\ref{eqn:prescribedC2}), region 1: $C=H^{-1}(v)+(K(1-v)-K(H^{-1}(1-u)))G(v)$ which proves (\ref{eqn:prescribedC1}), and region 4: $C=H^{-1}(v)+(K(1-v)+F(u))G(v)$ which proves (\ref{eqn:prescribedC3}). The final statement for $u+v>1$ follows from Theorem \ref{thm:copulaCsymmetry}. 
 \end{proof}
 \begin{figure}[H]\label{fig:SquarePartsPrescribed}
	\includegraphics[width=0.8\textwidth]{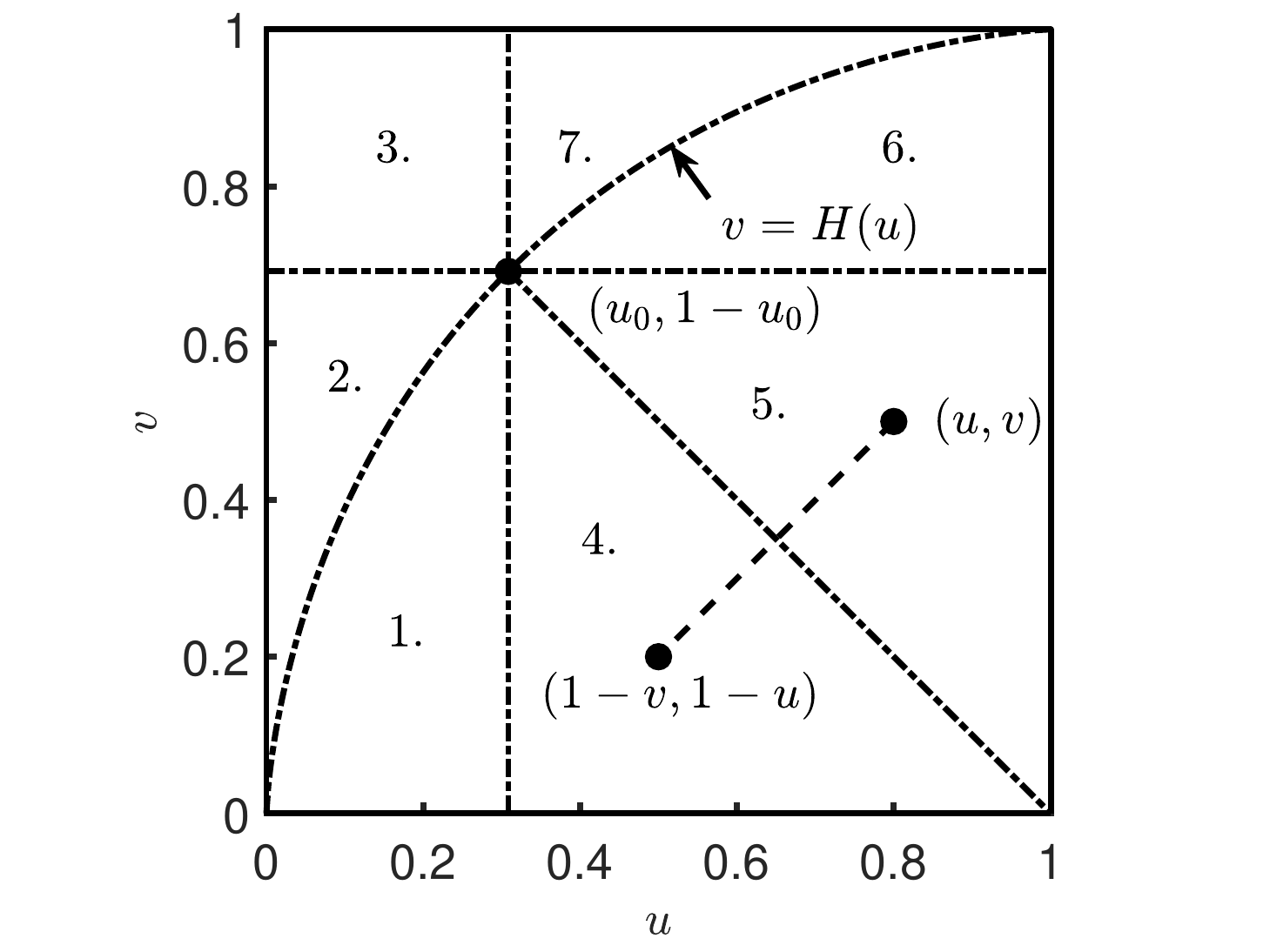}
	\caption{Subdivision of the unit square for piecewise definition of $p=C''_{uv}$ in Theorem \ref{thm:prescribed}.}
\end{figure}
 Equation (\ref{eqn:prescribedODE}) can be solved with the integrating factor method, and a positivity condition can be derived, analogous to Theorem \ref{thm:separable2}:
\begin{theorem}\label{thm:prescribed2}
	Assume that $K(u)$ is given by (\ref{eqn:K}), and $G$ satisfies the assumptions of Theorem \ref{thm:prescribed}. Then $F(u)$ satisfies (\ref{eqn:prescribedODE}) if and only if
	\begin{equation}\label{eqn:prescribedF}
		F(u)=-K(u)+G(1-u)\int_{u_0}^u\frac{1+H'(z)}{G(1-z)^2}dz.
	\end{equation}
	Moreover, if $F(u)$ is given by (\ref{eqn:prescribedF}), then
	\begin{equation}\label{eqn:prescribedFprime}
		F'(u)=G'(1-u)\left(\frac{L(u_0)}{G(1-u_0)^2}+\int_{u_0}^u\frac{1+L'(z)-H'(z)}{G(1-z)^2}dz\right)
	\end{equation}
	where $L$ is given by (\ref{eqn:L}).
	Finally, if $u^*\in[u_0,1]$, $L'(u)-H'(u)\leq -1$ for $u\in(u_0,u^*)$, $L'(u)-H'(u)\geq -1$ for $u\in(u^*,1)$ and
	\begin{equation}\label{eqn:prescribedPosCond}
		-\int_{u_0}^{u^*} \frac{1+L'(z)-H'(z)}{G(1-z)^2}dz\leq \frac{L(u_0)}{G(1-u_0)^2}
	\end{equation}
	then $F'(u)\geq0$ for $u\in(u_0,1)$.
\end{theorem}
\begin{proof}
	Multiplying (\ref{eqn:prescribedODE}) with the integrating factor $1/G(1-u)^2$ and integrating by parts (using $K(u_0)=0$) yields
	\begin{multline*}
		F(u)=G(1-u)\int_{u_0}^u \frac{1-G'(1-z)K(z)}{G(1-z)^2}\\
		=G(1-u)\left(\int_{u_0}^u\frac{dz}{G(1-z)^2}-\frac{K(u)}{G(1-u)}+\int_{u_0}^u\frac{K'(z)}{G(1-z)}\right)
	\end{multline*}
	so substituting 
	\begin{equation}\label{eqn:Kprime}
		K'(z)=\frac{H'(z)}{G(1-z)}
	\end{equation}
	according to (\ref{eqn:K}) yields (\ref{eqn:prescribedF}). Solving for $F'$ in (\ref{eqn:prescribedODE}): 
	\begin{equation}\label{eqn:prescribedODE2}
		F'(u)=\frac1{G(1-u)}-\frac1{L(u)}(K(u)+F(u))
	\end{equation}
	and substituting 
	\begin{equation}\label{eqn:KplusF1}
		K(u)+F(u)=\int_{u_0}^u\frac{1+H'(z)}{G(1-z)^2}
	\end{equation}
	according to (\ref{eqn:prescribedF}) yields 
	\begin{equation}\label{eqn:prescribedFprime2}
		F'(u)=G'(1-u)\left(\frac{L(u)}{G(1-u)^2}-\int_{u_0}^u\frac{1+H'(z)}{G(1-z)^2}dz\right)
	\end{equation}
	The identity (\ref{eqn:LbyGsquared}) yields 
	\begin{equation}
		\frac{L(u)}{G(1-u)^2}=\frac{L(u_0)}{G(1-u_0)^2}+\int_{u_0}^u\frac{2+L'(z)}{G(1-z)^2}dz
	\end{equation}
	 which proves (\ref{eqn:prescribedFprime}). Finally, by the assumptions, $u\mapsto -\int_{u_0}^u (1+L'(z)-H'(z))/G(1-z)^2 dz$ has its maximum for $u=u^*$, so it follows from (\ref{eqn:prescribedPosCond}) that $F'(u)\geq F'(u*)\geq 0$ for $u\in[u_0,1]$.
\end{proof}

\begin{example}
	If 
	\begin{equation}
		H(u)=\left\{
			\begin{matrix}
				(1-u_0)u/u_0&\text{ if }&u\leq u_0\\
				1-u_0(1-u)/(1-u_0)&\text{ if }&u>u_0
			\end{matrix}
		\right.,
	\end{equation}
 	$G(v)=v^k$ and $k\geq (1-u_0)/(1-2u_0)$, then (\ref{eqn:K}) yields
 	\begin{equation}
 		K(u)=\frac{((1-u)^{1-k}-(1-u_0)^{1-k})u_0}{(1-u_0)(k-1)},
 	\end{equation}
 	(\ref{eqn:prescribedF}) evaluates to
 	\begin{multline}
 		F(u)=\frac{(1-2u_0)k-(1-u_0)}{(2k-1)(k-1)(1-u_0)}(1-u)^{1-k} \\
 		-\frac{(1-u_0)^{1-2k}}{(2k-1)(1-u_0)}(1-u)^k
 		+ \frac{(1-u_0)^{1-k}u_0}{(k-1)(1-u_0)}.
 	\end{multline}
 	Moreover, $L(u)=(1-u)/k$, so $L'(u)-H'(u)=-1/k-u_0/(1-u_0)\geq -1$ if and only if $k\geq (1-u_0)/(1-2u_0)$, in which case $F'(u)$ is positive.  By theorem \ref{thm:prescribed} we obtain a two-parameter family of absolutely continuous copulas (with parameters $0<u_0<1/2$ and $k\geq (1-u_0)/(1-2u_0)$), with probability density supported on $v\leq H(u)$. Indeed, in this example $F'(u)$ can be computed explicitly:
 	\begin{equation}
 		F'(u)=\frac{((1-2u_0)k-(1-u_0))(1-u)^{-k}+k(1-u_0)^{1-2k}(1-u)^{k-1}}{(2k-1)(1-u_0)}
 	\end{equation}
 	and is strictly positive on $[u_0,1)$ if and only if the coefficient for $(1-u)^{-k}$ is positive, which is equivalent to $k\geq (1-u_0)/(1-2u_0)$.
 \end{example}
 \begin{example}\label{exm:Gaussian2}
	In this example we construct more solutions to Problem \ref{prb:standardNormalRestricted}, using Theorem \ref{thm:prescribed2}. Let $k\in\RR, k>1$ and $L(u)=(1-u)/k$. Then we obtain $G(v)=v^k/(1-u_0)^k$ and
	\begin{equation}
		K(u)=(1-u_0)^k\int_{u_0}^u\frac{H'(z)}{(1-z)^k}dz
	\end{equation}
	and
	\begin{equation}
		F(u)=-K(u)+(1-u_0)^k(1-u)^k\int_{u_0}^u\frac{1+H'(z)}{(1-z)^{2k}}dz
	\end{equation}
	where $H$ is given by (\ref{eqn:GaussianH}) and
	\begin{equation}\label{eqn:GaussianHprime}
		H'(z)=\frac1{\sqrt{2\pi}}\exp\left(-\Delta\left(\Phi^{-1}(u)+\frac{\Delta}2\right)\right).
	\end{equation}
	Since $L'(u)=-1/k$ and $H'$ decreasing we have $u^*$ satisfying the assumptions in Theorem \ref{thm:prescribed2} and determined by $H'(u^*)=1-1/k$. Solving this equation yields 
	\begin{equation}
		u^*=\Phi\left(-\frac{\sqrt{2\pi}}{\Delta}\left(1-\frac1k\right)-\frac{\Delta}2\right).
	\end{equation}
	Thus, $1-u^*=\Phi(\sqrt{2\pi}(1-1/k)/\Delta+\Delta/2)$, and also $1-u_0=\Phi(\Delta/2)$, and one can show that condition (\ref{eqn:prescribedPosCond}) is equivalent to
	\begin{equation}
		\int_{u_0}^{u^*}\frac{H'(z)}{(1-z)^{2k}}dz
		\leq\frac{ (1-u_0)^{1-2k}}{2k-1}+\left(1-\frac1k\right)(1-u^*)^{1-2k}
	\end{equation}
	so if $k$ satisfies this condition, an absolutely continuous copula is obtained. 
	\begin{figure}[H]\label{fig:Ckuv}
	$\begin{array}{rl}
	\includegraphics[width=0.5\textwidth]{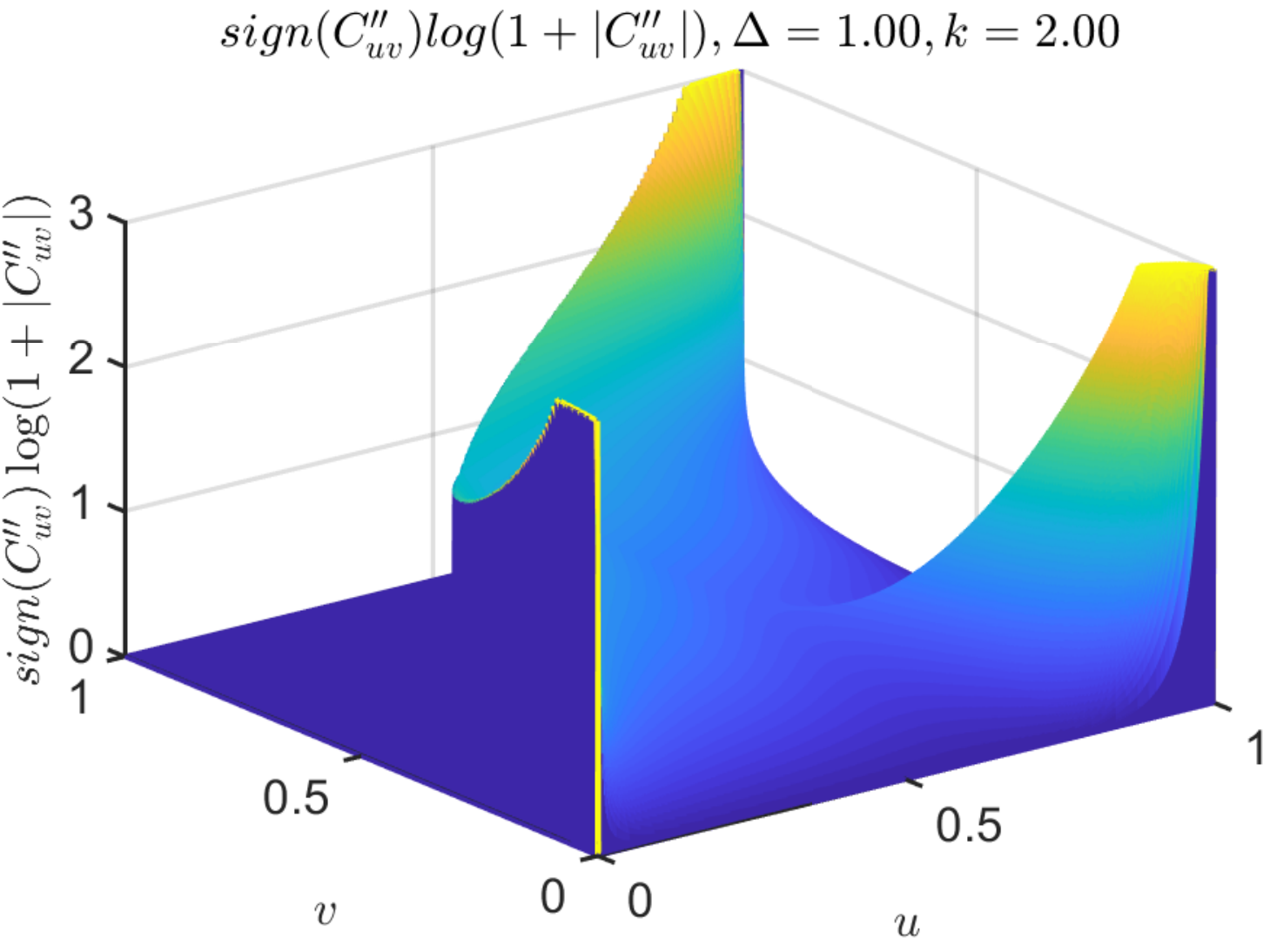}&
	\includegraphics[width=0.5\textwidth]{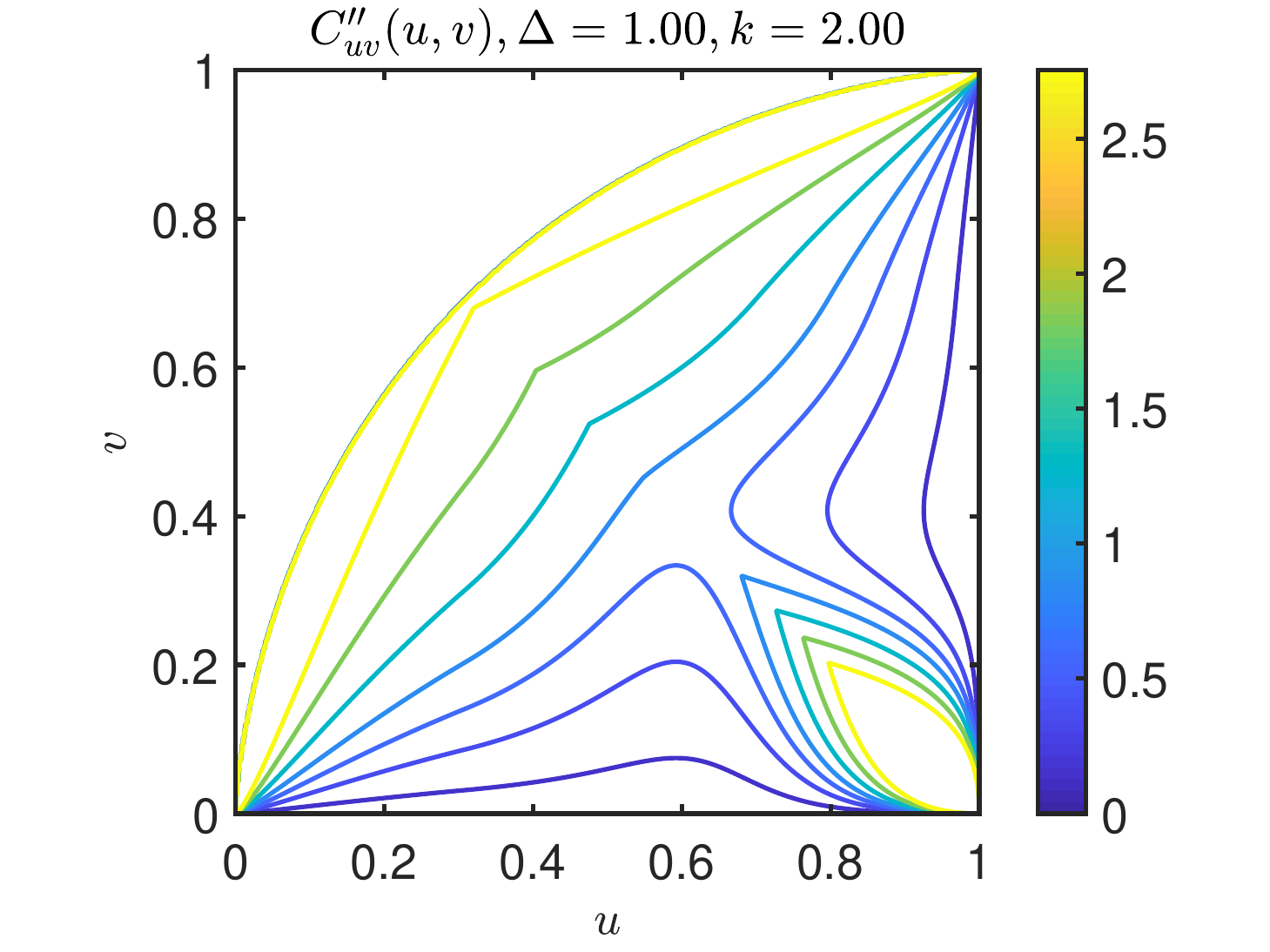}
	\end{array}$
	\caption{Copula density $C''_{uv}(u,v)$ for Example \ref{exm:Gaussian2}, $\Delta=1, k=2$. The density is discontinuous on the curve $v=H(u)$ and tends to infinity when approaching $(0,0)$, $(1,1)$ or $(1,0)$.}
	\end{figure}
	\begin{figure}[H]\label{fig:pk}
	$\begin{array}{rl}
	\includegraphics[width=0.5\textwidth]{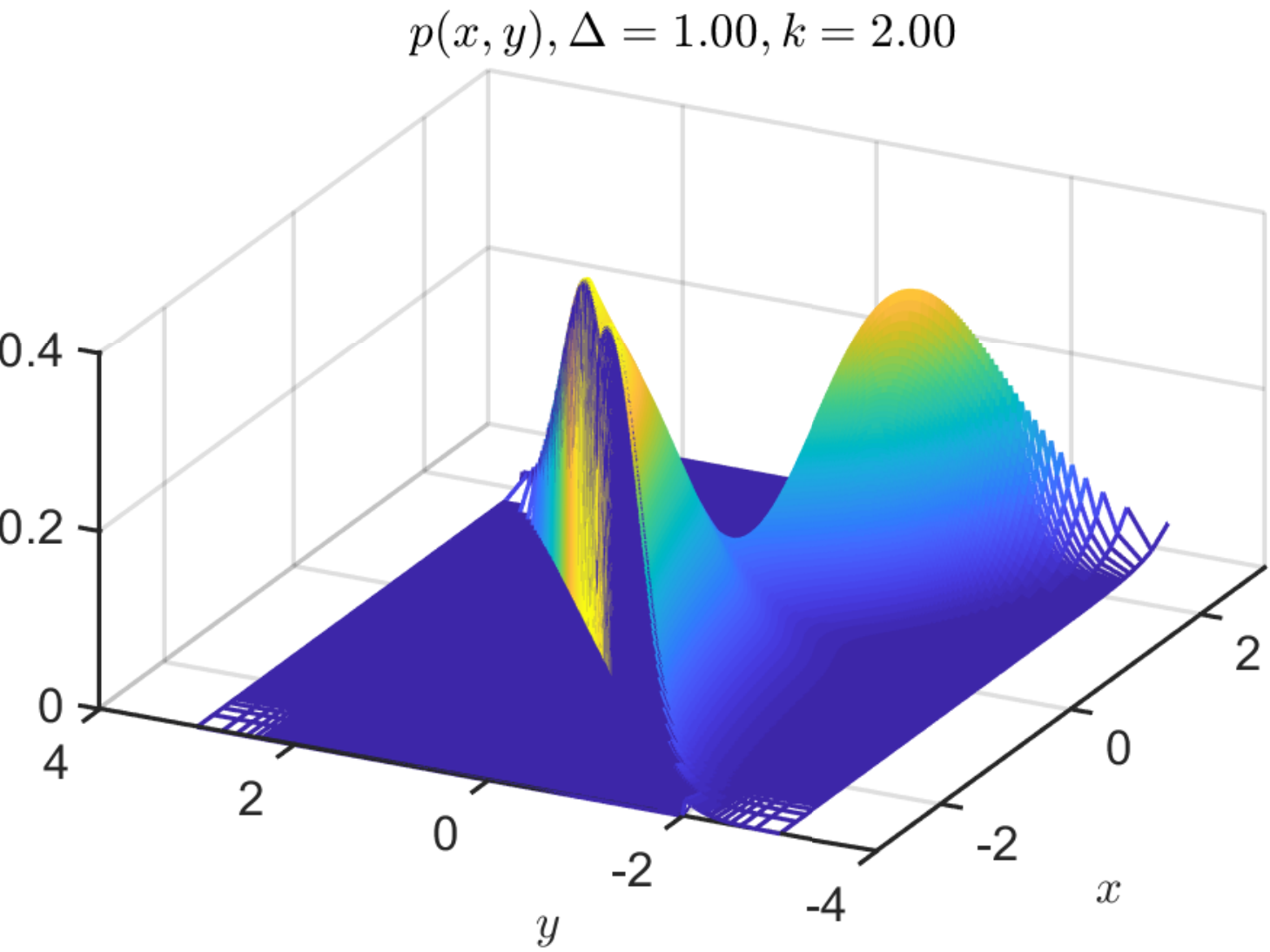}&
	\includegraphics[width=0.5\textwidth]{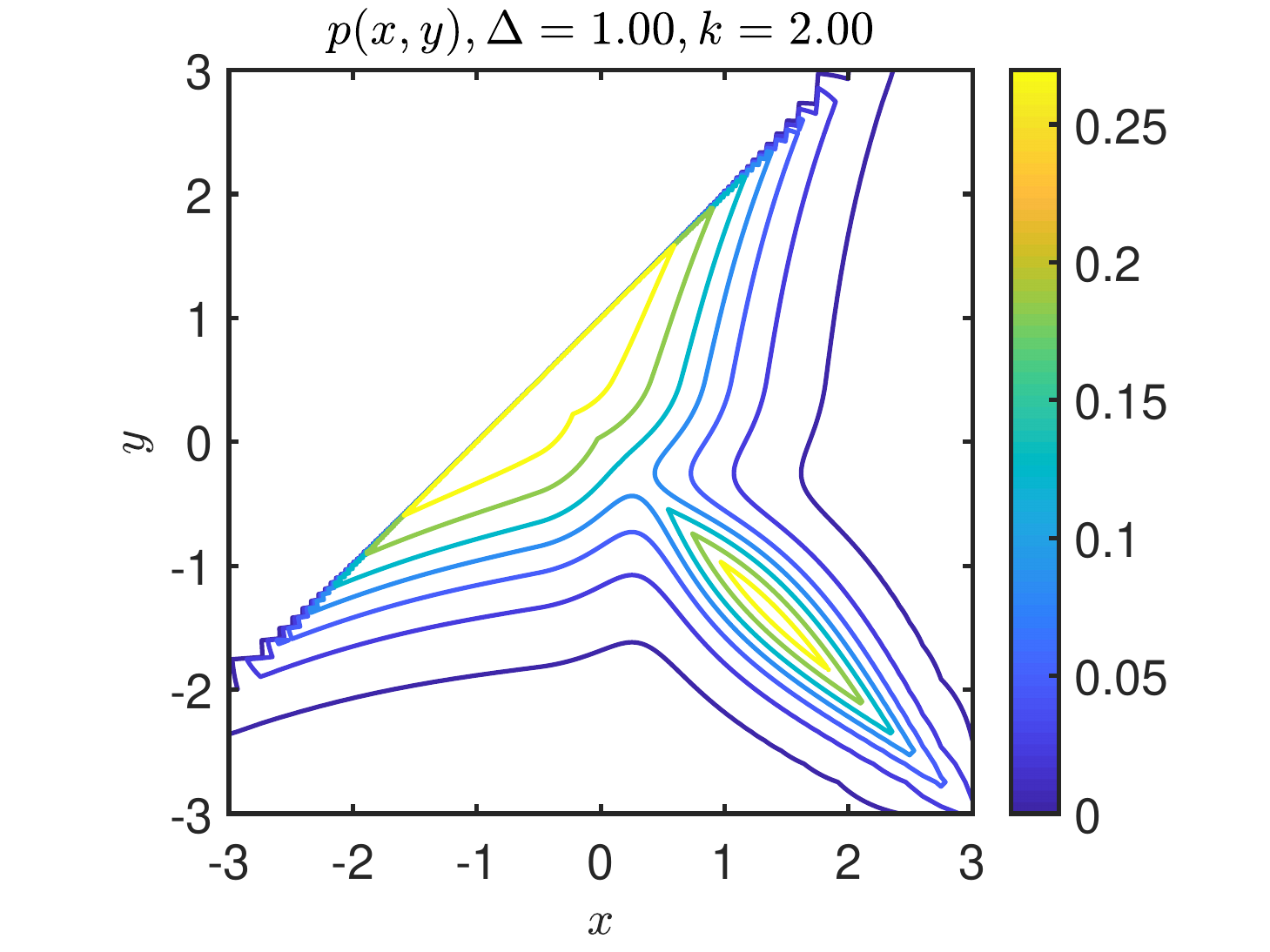}
	\end{array}$
	\caption{Probability density function $p(x,y)$ for Example \ref{exm:Gaussian2}, $\Delta=1, k=2$. The wiggles in the level curves at the upper right and lower left corners of right plot are numerical artifacts.}
	\end{figure}
\end{example}
 We have the following analogue of Theorem \ref{thm:separable3}. Here, given the opposite diagonal section $\omega$, the function $L$ is given by an \emph{integral equation} (\ref{eqn:prescribedLfromOmega}), (\ref{eqn:prescribed3GK}) below.
\begin{theorem}\label{thm:prescribed3}
	Suppose that $H, u_0$ satisfies (\ref{eqn:Hcond}) and (\ref{eqn:u0cond}). Suppose also that $L$ is a positive real--valued function defined on $[u_0,1]$ such that
	\begin{equation}\label{eqn:prescribed3Lcond1}
		\int_{u_0}^u \frac{dz}{L(z)}<\infty
	\end{equation}
	for $u\in[u_0,1)$ and
	\begin{equation}\label{eqn:prescribed3Lcond2}
		\lim_{u\to 1-}\int_{u_0}^u \frac{dz}{L(z)}=\infty.
	\end{equation}
	Let
	\begin{equation}\label{eqn:prescribedGfromL}
		G(v)=\exp\left(-\int_{u_0}^{1-v}\frac{dz}{L(z)}\right)
	\end{equation}
	 Moreover, let $K(u)$ and $F(u)$ be given by (\ref{eqn:K}) and (\ref{eqn:prescribedF}) and suppose that (\ref{eqn:prescribedPosCond}) holds true. Then $C$ given by (\ref{eqn:prescribedC1})-(\ref{eqn:prescribedC3}) and (\ref{eqn:oppositeSymmetryC}) is an absolutely continuous copula. Moreover, the opposite diagonal section (\ref{eqn:separableOppositeDiagonal})
	satisfies $\omega(u)=u$ for $u\in[0,u_0]$ and 
	\begin{equation}\label{eqn:prescribedLfromOmega}
		L(u)=\frac{2\omega(u)+G(1-u)K(u)}{1-\omega'(u)}
	\end{equation}
	for $u\in[u_0,1]$, where
	\begin{equation}\label{eqn:prescribed3GK}
		G(1-u)K(u)=\int_{u_0}^u\exp\left(-\int_z^u\frac{dw}{L(w)}\right)H'(z)dz
	\end{equation}
\end{theorem}
\begin{proof}
	The proof is similar to the proof of Theorem \ref{thm:separable3}, with some additional terms involving $K$. More precisely, (\ref{eqn:separable3FprimeG}) and (\ref{eqn:separable3FGprime}) are replaced by 
	\begin{equation}\label{eqn:prescribed3FprimeG}
		G(1-u)F'(u)=\frac{1+\omega'(u)-G'(1-u)K(u)}2
	\end{equation}
	and
	\begin{equation}\label{eqn:prescribed3FGprime}
		G'(1-u)F(u)=\frac{1-\omega'(u)-G'(1-u)K(u)}2.
	\end{equation}
	Solving for $F$ in (\ref{eqn:prescribed3FGprime}), differentiating and substituting for $F'$ in the left hand side of (\ref{eqn:prescribed3FprimeG}) yields
	\begin{multline}
		(1-\omega'(u))(1+L'(u))-\omega''(u)L(u)\\
		=1+\omega'(u)+K'(u)G(1-u)-K(u)G'(1-u)
	\end{multline}
	which is integrated to $(1-\omega'(u))L(u)=2\omega(u)+G(1-u)K(u)+$constant. The equation (\ref{eqn:prescribed3GK}) follows from (\ref{eqn:K}) and (\ref{eqn:prescribedGfromL}). For each fixed $z$, the integrand in (\ref{eqn:prescribed3GK}) is decreasing towards $0$ as $u\to1-$ in view of (\ref{eqn:prescribed3Lcond1}) and (\ref{eqn:prescribed3Lcond2}), so by the mononotone convergence theorem, $\lim_{u\to1-}G(1-u)K(u)=0$. Hence the constant of integration is zero, which proves (\ref{eqn:prescribedLfromOmega}).
\end{proof}
\begin{proof}[Proof of Theorem \ref{thm:main}]
	Since $1+L'(z)-H'(z)\equiv 0$, $L(u_0)=H(u_0)-u_0=1-2u_0$ and $G(1-u_0)=1$ we have by (\ref{eqn:prescribedFprime})
	\begin{equation}\label{eqn:mainFprime}
		F'(u)=(1-2u_0)G'(1-u)
	\end{equation}
	so $F(u)$ satisfies (\ref{eqn:mainF}). Solving for $K$ in (\ref{eqn:prescribedODE}) yields
	\begin{equation}\label{eqn:prescribedODE3}
		K(u)=\frac1{G'(1-u)}-\frac{G(1-u)}{G'(1-u)}F'(u)-F(u)
	\end{equation}
	and substituting (\ref{eqn:mainF}) and (\ref{eqn:mainFprime}) in (\ref{eqn:prescribedODE3}) implies that $K(u)$ satisfies (\ref{eqn:mainK}).
\end{proof}
\section{Sampling} 
\label{sec:sampling}
To sample from a two--dimensional copula $C(u,v)$ we use the conditional density $C'_u$ of Corollary \ref{cor:prescribed2} in the following way (cf. \cite[Chap. 2.9]{Nelsen1999}): First sample $U,T$, independently from $U(0,1)$. Then for each $T_i,T_i$ let $V_i$ satisfy $T_i=C'_u(U_i,V_i)$. Then $(U_i,V_i)$ is distributed according to $C(u,v)$. For sampling from the copula, the following corollary is useful:

 \begin{corollary}\label{cor:prescribed2}
 	Suppose that $C(u,v)$ is an absolutely continuous copula given by Theorem \ref{thm:prescribed} and $F,G,K$ defined accordingly. Then $C'_u(u,v)$ is given by the following formulas:
 	\begin{enumerate}
 		\item If $0\leq u\leq u_0$ and $0<v<H(u)$ then
 		\begin{equation}\label{eqn:prescribedCu1b}
 			C'_u(u,v)=G(v)/G(H(u))
 		\end{equation}
 		\item If $0\leq u\leq 1$ and $H(u)\leq v\leq 1$ then 
 		\begin{equation}\label{eqn:prescribedCu2b}
 			C'_u(u,v)=1
 		\end{equation}
 		\item If $u_0<u<1$ and $0<v\leq 1-u$ then
 		\begin{equation}\label{eqn:prescribedCu3b}
 			C'_u(u,v)=(1-G'(1-u)(K(u)+F(u)))G(v)/G(1-u)
 		\end{equation}
 		\item If $u_0<u<1$ and $1-u<v\leq 1-u_0$ then
 		\begin{equation}\label{eqn:prescribedCu4b}
 		 	C'_u(u,v)=1-G'(1-u)(K(u)+F(1-v))
 		 \end{equation}
 		 \item If $u_0<u<1$ and $1-u_0<v\leq H(u)$ then
 		 \begin{equation}\label{eqn:prescribedCu5b}
 		 	C'_u(u,v)=1-G'(1-u)(K(u)-F(1-H(1-v)))
 		 \end{equation}
 		 \item If $u_0<u<1$ and $H(u)<v<1$ then $C'_u(u,v)=1$.
 	\end{enumerate}
 \end{corollary}
\begin{proof}
	Follows from the equations for $C'_u$ in the proof of Theorem \ref{thm:prescribed}, and equations (\ref{eqn:symmetricH}), (\ref{eqn:prescribedODE}).
\end{proof}
Figure \ref{fig:sampling} illustrates sampling in Example \ref{exm:Gaussian}.
\begin{figure}[H]\label{fig:sampling}
	$\begin{array}{rl}
	\includegraphics[width=0.5\textwidth]{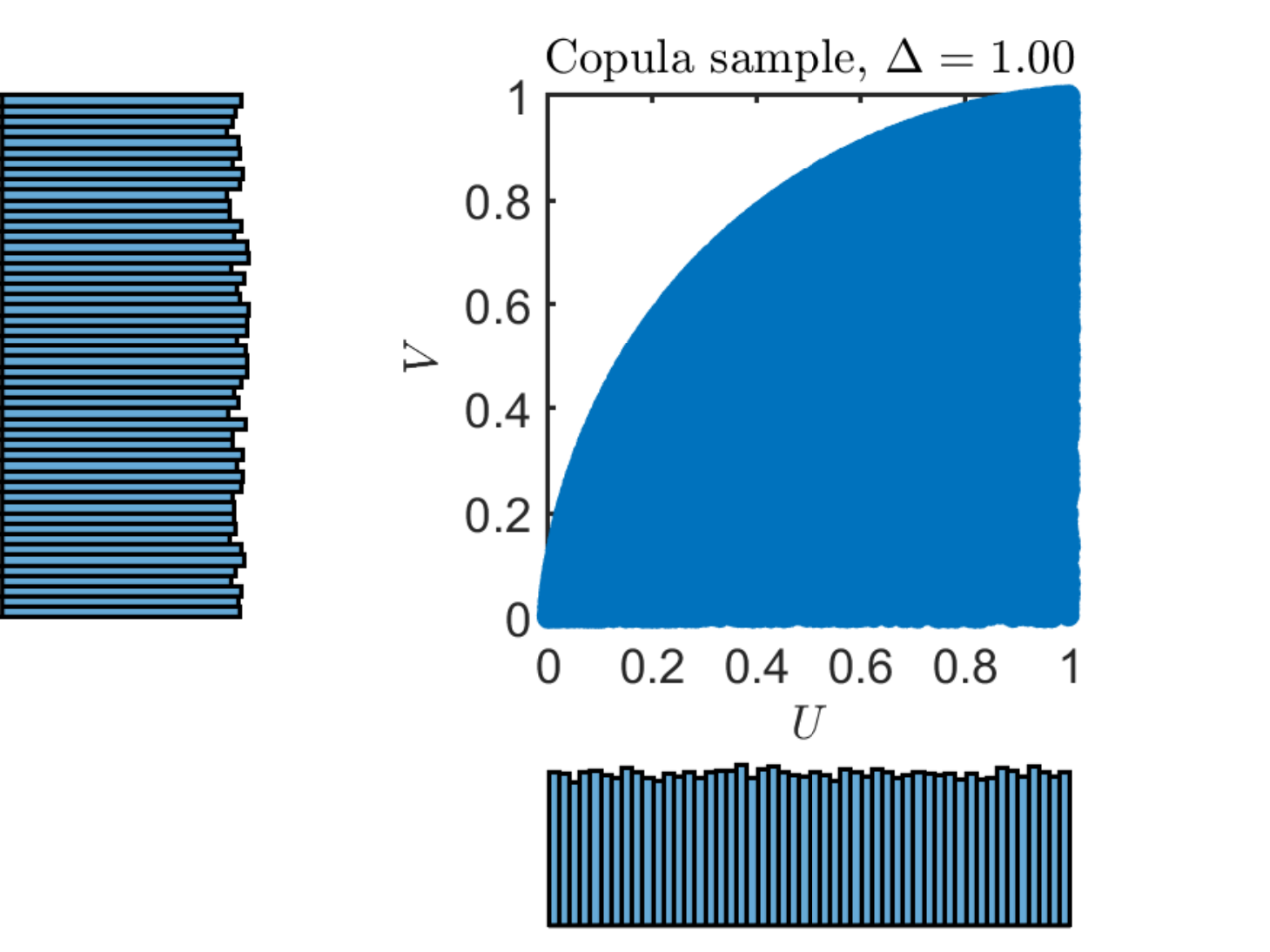}&
	\includegraphics[width=0.5\textwidth]{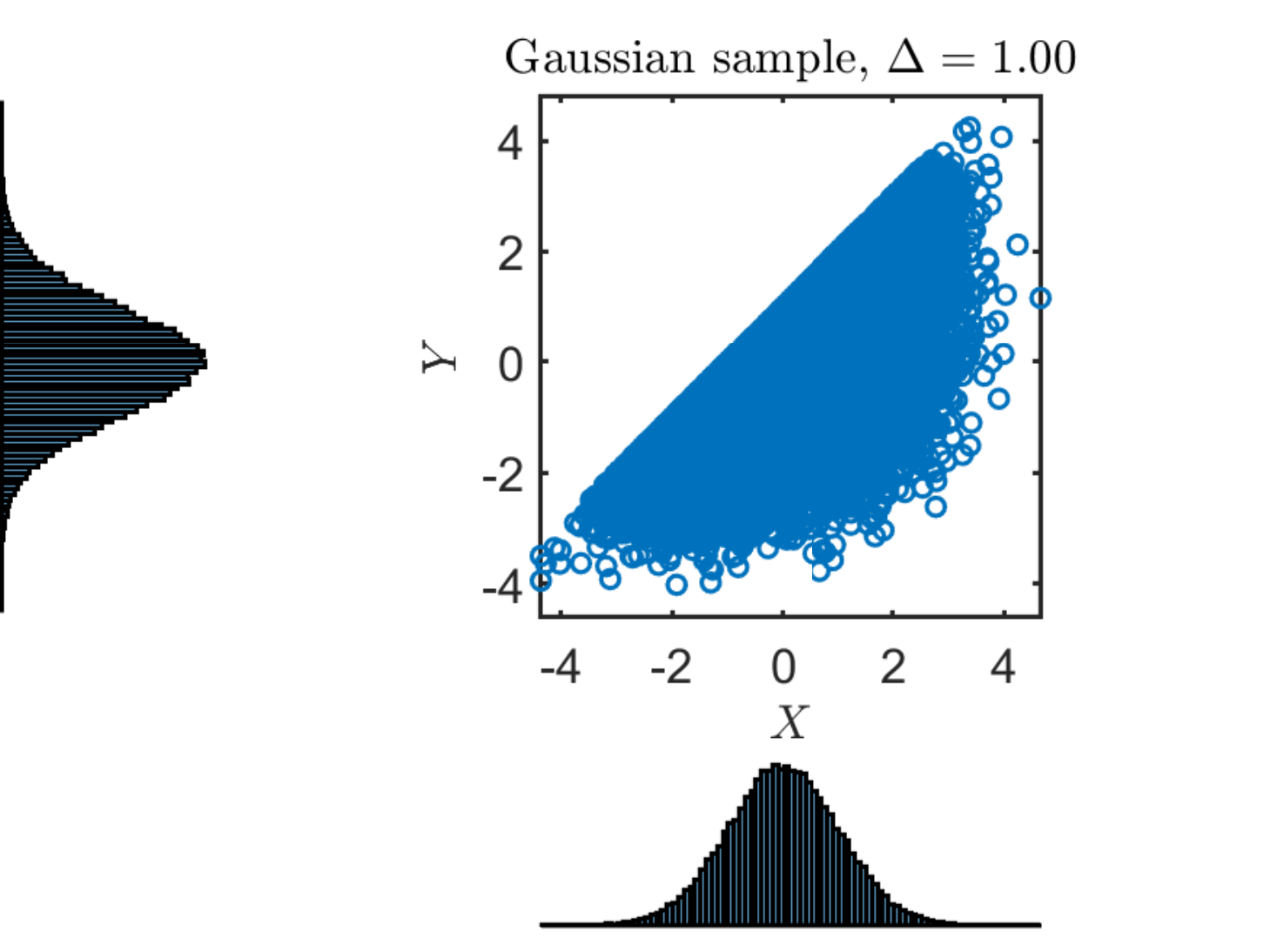}
	\end{array}$
	\caption{Samples from distributions in Example \ref{exm:Gaussian}, sample size $10^5$.}
	\end{figure}

\section{Application to toxicological probit models}\label{sec:Toxicology}
The probit model is the standard statistical method for estimating the injury outcome of a population exposed to a toxic substance. It originates from an analysis on the effect of pesticides conducted by Bliss in 1934 \cite{Bliss1938}. The methodology was later cast in a more rigid mathematic formulation by Finney \cite{FinneyTattersfield1952,Finney1977}. It has since then been used frequently in toxicological assessments of the injury outcome when a population has been exposed to dangerous chemicals \cite{McBride2001,Bjornham_etal2017,BurmanJonsson2015,Hauptmanns2005,Lovreglio_etal2016,Stage2004,Haghnazarloo2015}. In short, the probit model operates as follows. The exposure concentration $c(t)$ is integrated over time to yield \emph{probit values}
\begin{equation}\label{eqn:Gamma_i}
	\Gamma_i(t)=\alpha_i+\beta_i\log\left(\int_0^t c(t)^{n_i} dt\right).
\end{equation}
The fraction of the population that has attainted the injury at time $t$ is then estimated by
\begin{equation}\label{eqn:InjuryProbability}
  	\Phi(\Gamma_i(t))
\end{equation}
where $\alpha_i,\beta_i, n_i$ are model parameters associated with the substance, and $\Phi$ is the CDF for a standard normal variable. There are often several levels of injury outcome used in toxicology, e.g., \emph{light injury, severe injury} and \emph{death}. These different injury levels are indexed by $i=1,2,...$ in equations (\ref{eqn:Gamma_i})-(\ref{eqn:InjuryProbability}). The fraction of the population that obtains an injury increases continuously with growing exposure due to the individual variation of the toxic susceptiblity within the population. It is believed that modeling this variation improves the quantitative toxicological risk assessment, cf. \cite{HattisBanatiGoble1999}.

A population that is not resolved on an individual level is referred to as a \emph{macroscopic population} and can be described as a density field. In contrast, a population can be described as a set of discrete individuals, referred to as \emph{agents}. A model that uses this type of population representation is called a \emph{microscale model} or an \emph{agent--based model}. In an agent--based toxicological model, see for example \cite{Lovreglio_etal2016}, the overall population statistics is obtained from the set of agents that are exposed to the toxic substance. In such a setting, individual probit values $\Gamma_i(t)$, acquired by exposure to individual model concentrations $c(t)$, are computed for each agent. In the transition from a macroscopic population to an agent--based population, it is convenient to distribute individual threshold values, $\gamma_i$, for the probit values to all agents representing their susceptibilities. Thus, when an agent has been exposed to a concentration yielding a probit value exceeding the corresponding threshold value, the agent has acquired that injury. Every agent is attributed one threshold value for each injury level. These threshold values are drawn from a standard normal distribution to maintain the overall probability distribution for the entire population. This method implies that the injury outcome of the agent--based population approaches asymptotically that of the macrosopic population (with static populations) when the number of agents increases. An advantage with an agent--based population is that the agents may have individual properties including their movement patterns. In a dynamic simulation, each agent follows its individual spacetime path, passing through concentration fields, and thereby proceeds through some or all of the injury stages, transiting successive injury stages when the agent's increasing probit functions $\Gamma_i(t)$ pass their threshold values $\gamma_i$. As mentioned, the individual toxic susceptibility thresholds $\gamma_i$ are random variables and must obey the requirement
\begin{equation}
	P(\gamma_i\leq \Gamma)=\Phi(\Gamma)
\end{equation}
We propose that the $\gamma_1, \gamma_2,...$ are  modeled as a discrete time Markov process with absolutely continuous transition densities $p_{i+1\mid i}$, so by the Markov property, the joint density $p$ is
 \begin{equation}
 	p(\gamma_1,...,\gamma_n)=p_1(\gamma_1)p_{2\mid1}(\gamma_2\mid\gamma_1)p_{3\mid2}(\gamma_3\mid\gamma_2)...p_{n\mid n-1}(\gamma_n\mid\gamma_{n-1}).
 \end{equation}
However, there is a potential pitfall: the injury stages must be passed in the correct order. Therefore, it must be true with probability one that if an injury level is acquired, then also the previous injury level is acquired, i.e.
\begin{equation}\label{eqn:pdfSupportConclusion}
 	\gamma_{i+1}\leq \Gamma_{i+1}(t)\implies \gamma_i\leq \Gamma_i(t).
 \end{equation}
 Therefore, the transition densities $p_{i+1\mid i}$ must satisfy
\begin{equation}
	p_{i+1\mid i}(\gamma_{i+1}\mid\gamma_i)=0\text{ if }\gamma_{i+1}\leq\Gamma_{i+1}(t)\text{ and } \gamma_i>\Gamma_i(t).
\end{equation}
 This imposes a restriction on the support of the joint probability density of $(\gamma_i,\gamma_{i+1})$, which we need to investigate in order to ensure that the model is consistent. To this end, we need to relate possible values of $\Gamma_i(t),\Gamma_{i+1}(t)$ for all possible exposures $c(t)$, $t\geq 0$. This can be done in terms of
 \begin{equation}\label{eqn:toxicLoad}
 	\frac{\Gamma_i(t)-\alpha_i}{\beta_i}=\log\left(\int_0^tc^{n_i}dt\right)
 \end{equation}
 according to the following lemma:

 \begin{lemma}\label{lem:toxicLoadInequalities}
 	Assume that $n\geq m>0$ and $c\geq 0, t>0$. Then
 		\begin{equation}\label{eqn:toxicLoadInequality1}
 			\log\left(\int_0^t c^m dt\right)\leq \frac mn\log\left(\int_0^t c^n dt\right)+\left(1-\frac mn\right)\log(t)
 		\end{equation}
 		and
 		\begin{equation}\label{eqn:toxicLoadInequality2}
 			\log\left(\int_0^t c^n dt\right)\leq \log\left(\int_0^t c^m dt\right)+(n-m)\log\left(\max_{[0,t]}c\right)
 		\end{equation}
 	Moreover, the inequalities are sharp: if $c(t)=$constant, then equalities holds in the inequalities above. 
 \end{lemma}
 \begin{proof}
 	Apply H\"older's inequality $\int fg dt\leq\left(f^p dt\right)^{1/p}\left(g^q dt\right)^{1/q}$ and the elementary estimate $\int f^p dt\leq(\max f)^{p-1}\int f dt$ with $f=c^m$, $g=1$ and $p=n/m$. 
 \end{proof}
The following theorems provide sufficient conditions for (\ref{eqn:pdfSupportConclusion}), and necessary compatibility conditions for the probit parameters $\alpha,\beta,n$.
\begin{theorem}\label{thm:pdfSupport1}
 	Assume that $\Gamma_i(t), \Gamma_{i+1}(t)$ are probit functions defined by (\ref{eqn:Gamma_i}), and $n_{i+1}\leq n_i$. Also assume that $(\gamma_i,\gamma_{i+1})$ is a bivariate random variable such that
 		 \begin{equation}\label{eqn:pdfSupport1}
 			\frac{\gamma_{i+1}-\alpha_{i+1}}{\beta_{i+1}}
 			\geq \frac{n_{i+1}}{n_i}\frac{\gamma_i-\alpha_i}{\beta_i}
 			+ \left(1-\frac{n_{i+1}}{n_i}\right)\log(t)
 		\end{equation}
	almost surely. Then $\gamma_{i+1}\leq\Gamma_{i+1}(t)\implies\gamma_i\leq\Gamma_i(t)$ almost surely. Moreover, there exists standard normal $\gamma_i, \gamma_{i+1}$ satisfying (\ref{eqn:pdfSupport1}) if and only if 
	\begin{equation}\label{eqn:probitCond1beta}
			n_{i+1}\beta_{i+1}=n_i\beta_i
	\end{equation}
	and
	\begin{equation}\label{eqn:probitCond1Delta}
			\Delta_i\equiv\alpha_i-\alpha_{i+1}-\beta_{i+1}\left(1-\frac{n_{i+1}}{n_i}\right)\log t\geq0,
	\end{equation}
	and then if $\Delta_i>0$ there exists $(\gamma_{i},\gamma_{i+1})$ with absolutely continuous joint density.
 \end{theorem}
 \begin{proof} Assume that $\gamma_{i+1}\leq\Gamma_{i+1}(t)$. Then we get by (\ref{eqn:toxicLoad}), (\ref{eqn:toxicLoadInequality1}) with $m=n_{i+1}$, $n=n_i$,  and (\ref{eqn:pdfSupport1}) that
 	\begin{multline}
 	\frac{n_{i+1}}{n_i}\frac{\Gamma_i(t)-\alpha_i}{\beta_i}+\left(1-\frac{n_{i+1}}{n_i}\right)\log(t)\\
 		\geq\frac{\Gamma_{i+1}(t)-\alpha_{i+1}}{\beta_{i+1}}
 		\geq \frac{\gamma_{i+1}-\alpha_{i+1}}{\beta_{i+1}}\\
 		\geq \frac{n_{i+1}}{n_i}\frac{\gamma_i-\alpha_i}{\beta_i}+\left(1-\frac{n_{i+1}}{n_i}\right)\log(t)
 	\end{multline}
 	i.e., $\Gamma_i(t)\geq \gamma_i$, which proves the first part. The second part follows from Proposition \ref{pro:linearConstraintStandardNormal}, since equation (\ref{eqn:pdfSupport1}) is equivalent to equation (\ref{eqn:linearConstraintStandardNormal}) with $X=-\gamma_i$, $Y=-\gamma_{i+1}$, $a=(\beta_{i+1}n_{i+1})/(\beta_in_i)$ and
 	\begin{equation*}\label{eqn:pdfSupport1Delta}
 		\Delta=\frac{\beta_{i+1}n_{i+1}}{\beta_in_i}\alpha_i-\alpha_{i+1}
 		-\beta_{i+1}\left(1-\frac{n_{i+1}}{n_i}\right)\log(t),
 	\end{equation*}
 	and $a=1, \Delta\geq0$ is equivalent to equations (\ref{eqn:probitCond1beta}), (\ref{eqn:probitCond1Delta}).
\end{proof}
\begin{theorem}\label{thm:pdfSupport2}
 	Assume that $\Gamma_i(t), \Gamma_{i+1}(t)$ are probit functions defined by (\ref{eqn:Gamma_i}), and $n_{i+1}\geq n_i$. Also assume that $(\gamma_i,\gamma_{i+1})$ is a bivariate random variable such that
 		 \begin{equation}\label{eqn:pdfSupport2}
 			\frac{\gamma_{i+1}-\alpha_{i+1}}{\beta_{i+1}}\geq \frac{\gamma_i-\alpha_i}{\beta_i}+(n_{i+1}-n_i)\log\left(\max_{[0,t]}c\right)
 		\end{equation}
 	almost surely. Then $\gamma_{i+1}\leq\Gamma_{i+1}(t)\implies\gamma_i\leq\Gamma_i(t)$ almost surely. Moreover, there exist standard normal $\gamma_i, \gamma_{i+1}$ satisfying (\ref{eqn:pdfSupport2}) if and only if 
	\begin{equation}\label{eqn:probitCond2beta}
			\beta_{i+1}=\beta_i
		\end{equation}
	and
	\begin{equation}\label{eqn:probitCond2Delta}
			\Delta_i\equiv\alpha_i-\alpha_{i+1}-\beta_i(n_{i+1}-n_i)\log\left(\max_{[0,t]}c\right)\geq0, 
		\end{equation}
	and then if $\Delta_i>0$ there exists $(\gamma_{i},\gamma_{i+1})$ with absolutely continuous joint density.
 \end{theorem}
 \begin{proof}[Proof of Theorem \ref{thm:pdfSupport2}]
 	Assume that $\gamma_{i+1}\leq\Gamma_{i+1}(t)$. Then we get by (\ref{eqn:toxicLoad}), (\ref{eqn:toxicLoadInequality2}) with $m=n_i$, $n=n_{i+1}$ and (\ref{eqn:pdfSupport2}) that
 	\begin{multline}
 	\frac{\Gamma_i(t)-\alpha_i}{\beta_i}+(n_{i+1}-n_i)\log\left(\max_{[0,t]}c\right)\\
 		\geq\frac{\Gamma_{i+1}(t)-\alpha_{i+1}}{\beta_{i+1}}
 		\geq \frac{\gamma_{i+1}-\alpha_{i+1}}{\beta_{i+1}}\\
 		\geq \frac{\gamma_i-\alpha_i}{\beta_i}+(n_{i+1}-n_i)\log\left(\max_{[0,t]}c\right)
 	\end{multline}
 	i.e., $\Gamma_i(t)\geq \gamma_i$, which proves the first part. The second part follows from Proposition \ref{pro:linearConstraintStandardNormal}, since equation (\ref{eqn:pdfSupport2}) is equivalent to equation (\ref{eqn:linearConstraintStandardNormal}) with $X=-\gamma_i$, $Y=-\gamma_{i+1}$, $a=\beta_{i+1}/\beta_i$ and 
 	\begin{equation*}
 		\Delta=\frac{\beta_{i+1}}{\beta_i}\alpha_i-\alpha_{i+1}-\beta_{i+1}\left(n_{i+1}-n_i\right)\log\left(\max_{[0,t]}c(t)\right), 
 	\end{equation*}
 	and $a=1, \Delta\geq0$ is equivalent to equations (\ref{eqn:probitCond2beta}), (\ref{eqn:probitCond2Delta}).
 \end{proof}
\begin{remark}
	Note that if $n_{i+1}=n_i$, then the compatibility conditions (\ref{eqn:probitCond1beta}), (\ref{eqn:probitCond1Delta}) and (\ref{eqn:probitCond2beta}), (\ref{eqn:probitCond2Delta}) in the preceding theorems involve only the probit coefficients $\alpha,\beta,n$, not $t$ or $\max c$.
\end{remark}
\bibliographystyle{plain}
\bibliography{copula}
\end{document}